\documentclass[reqno]{amsart}
\usepackage{amsmath, amsthm, amssymb, amstext}

\usepackage{hyperref,xcolor}
\hypersetup{pdfborder={0 0 0},colorlinks}
\usepackage{enumitem}
\allowdisplaybreaks
\usepackage{todonotes}

\newtheorem{theorem}{Theorem}
\newtheorem{remark}[theorem]{Remark}
\newtheorem{lemma}[theorem]{Lemma}
\newtheorem{proposition}[theorem]{Proposition}

\newtheorem{definition}[theorem]{Definition}

\newcommand{\N}{\mathbb{N}}
\newcommand{\R}{\mathbb{R}}

\renewcommand{\l}{\left}
\renewcommand{\r}{\right}

\numberwithin{theorem}{section}
\numberwithin{equation}{section}

\title[Degenerate Kirchhoff double phase problems]{Least energy sign-changing solution for degenerate Kirchhoff double phase problems}

\author[\'{A}. Crespo-Blanco]{\'{A}ngel Crespo-Blanco}
\address[\'{A}. Crespo-Blanco]{Technische Universit\"{a}t Berlin, Institut f\"{u}r Mathematik, Stra\ss e des 17.\,Juni 136, 10623 Berlin, Germany}
\email{crespo@math.tu-berlin.de}

\author[L. Gasi\'{n}ski]{Leszek Gasi\'{n}ski}
\address[L. Gasi\'{n}ski]{Department of Mathematics, University of the National Education Commission, Krakow, Podchorazych 2, 30-084 Krakow, Poland}
\email{leszek.gasinski@up.krakow.pl}

\author[P. Winkert]{Patrick Winkert}
\address[P. Winkert]{Technische Universit\"{a}t Berlin, Institut f\"{u}r Mathematik, Stra\ss e des 17.\,Juni 136, 10623 Berlin, Germany}
\email{winkert@math.tu-berlin.de}


\begin{document}

\begin{abstract}
	In this paper we study the following nonlocal Dirichlet equation of double phase type
	\begin{align*}
		-\psi \left [ \int_\Omega \left ( \frac{|\nabla u |^p}{p} + \mu(x) \frac{|\nabla u|^q}{q}\right)\,\mathrm{d} x\right] \mathcal{G}(u) = f(x,u)\quad \text{in } \Omega, \quad u = 0\quad  \text{on } \partial\Omega,
	\end{align*}
	where $\mathcal{G}$ is the double phase operator given by
	\begin{align*}
		\mathcal{G}(u)=\operatorname{div} \left(|\nabla u|^{p-2}\nabla u + \mu(x) |\nabla u|^{q-2}\nabla u \right)\quad u\in W^{1,\mathcal{H}}_0(\Omega),
	\end{align*}
	$\Omega\subseteq \mathbb{R}^N$, $N\geq 2$, is a bounded domain with Lipschitz boundary $\partial\Omega$, $1<p<N$, $p<q<p^*=\frac{Np}{N-p}$, $0 \leq \mu(\cdot)\in L^\infty(\Omega)$, $\psi(s) = a_0 + b_0 s^{\vartheta-1}$ for $s\in\mathbb{R}$,  with $a_0 \geq 0$, $b_0>0$ and $\vartheta \geq 1$, and $f\colon\Omega\times\mathbb{R}\to\mathbb{R}$ is a Carath\'{e}odory function that grows superlinearly and subcritically. We prove the existence of two constant sign solutions (one is positive, the other one negative) and of a sign-changing solution which turns out to be a least energy sign-changing solution of the problem above. Our proofs are based on variational tools in combination with the quantitative deformation lemma and the Poincar\'{e}-Miranda existence theorem.
\end{abstract}

\maketitle

\section{Introduction and main results}
Given a bounded domain $\Omega\subseteq \R^N$, $N\geq 2$, with Lipschitz boundary $\partial\Omega$, we study double phase problems with a nonlocal Kirchhoff term of the form
\begin{equation}
	\label{problem}
	\begin{aligned}
		-\psi \left [ \int_\Omega \left ( \frac{|\nabla u |^p}{p} + \mu(x) \frac{|\nabla u|^q}{q}\right)\,\mathrm{d} x\right] \mathcal{G}(u) & = f(x,u)\quad &  & \text{in } \Omega,\\
		u& = 0           &  & \text{on } \partial\Omega,
	\end{aligned}
\end{equation}
where
\begin{align}\label{operator_double_phase}
	\mathcal{G}(u)=\operatorname{div} \left(|\nabla u|^{p-2}\nabla u + \mu(x) |\nabla u|^{q-2}\nabla u \right)\quad \text{for }u\in W^{1,\mathcal{H}}_0(\Omega)
\end{align}
is the double phase operator, $1<p<N$, $p<q<p^*=\frac{Np}{N-p}$, $0 \leq \mu(\cdot)\in L^\infty(\Omega)$, $\psi(s) = a_0 + b_0 s^{\vartheta-1}$ for $s\in\R$ with $a_0 \geq 0$, $b_0>0$ and $\vartheta \geq 1$, and $f\colon\Omega\times\R\to\R$ is a Carath\'{e}odory function that grows superlinearly and subcritically, see the precise assumptions \eqref{H3} below.

Problem \eqref{problem} combines a nonlocal Kirchhoff term with a nonhomogeneous operator of double phase type. Note that the operator given in \eqref{operator_double_phase} is related to the energy functional
\begin{align}\label{integral_minimizer}
	\omega \mapsto \int_\Omega \big(|\nabla  \omega|^p+\mu(x)|\nabla  \omega|^q\big)\,\mathrm{d} x,
\end{align}
which was first introduced by Zhikov \cite{Zhikov-1986} in 1986. Such functional can be used to describe models for strongly anisotropic materials in the context of homogenization and elasticity. In fact, the hardening properties of strongly anisotropic materials change point by point and the modulating coefficient $\mu(\cdot)$ helps to regulate the mixture of two different materials with hardening powers $p$ and $q$. Also, functionals of the shape \eqref{integral_minimizer} have several mathematical applications in the study of duality theory and of the Lavrentiev gap phenomenon, see Zhikov \cite{Zhikov-1995,Zhikov-2011}. Regularity properties for local minimizers of \eqref{integral_minimizer} have been first published in the papers by Baroni-Colombo-Mingione \cite{Baroni-Colombo-Mingione-2015,Baroni-Colombo-Mingione-2018} and Colombo-Mingione \cite{Colombo-Mingione-2015a,Colombo-Mingione-2015b}. It should be noted that \eqref{integral_minimizer} belongs to the class of the integral functionals with nonstandard growth condition as a special case of the groundbreaking works by Marcellini \cite{Marcellini-1991,Marcellini-1989}. We also refer to the recent papers by Cupini-Marcellini-Mascolo \cite{Cupini-Marcellini-Mascolo-2023} and Marcellini \cite{Marcellini-2023} with $u$-dependence.

Another interesting phenomenon in problem \eqref{problem} is the occurrence of the nonlocal Kirchhoff term which generalizes a model introduced by Kirchhoff \cite{Kirchhoff-1883} of the form
\begin{align*}
	\rho \frac{\partial^2 u}{\partial t^2}-\left(\frac{\rho_0}{h}+\frac{E}{2L}\int_0^L \left|\frac{\partial u}{\partial x}\right|^2 \,\mathrm{d} x\right)\frac{\partial^2 u}{\partial x^2}=0
\end{align*}
with constants $\rho, \rho_0, h, E$ and $L$. Such model can be seen as a generalization of the classical D'Alembert wave equation. In general, Kirchhoff-type problems have a strong background in several applications in physics. After the outstanding work of Lions \cite{Lions-1978} about the abstract framework of Kirchhoff-type problems, several works have been published with different structure and various techniques. In the extensive literature on Kirchhoff problems, we refer, for example, to the papers of Autuori-Pucci-Salvatori \cite{Autuori-Pucci-Salvatori-2010}, D'Ancona-Spagnolo \cite{DAncona-Spagnolo-1992}, Figueiredo \cite{Figueiredo-2013}, Fiscella \cite{Fiscella-2019}, Fiscella-Valdinoci \cite{Fiscella-Valdinoci-2014}, Gasi\'{n}ski-Santos J\'{u}nior \cite{Gasinski-Santos-Junior-2020}, Mingqi-R\u{a}dulescu-Zhang \cite{Mingqi-Radulescu-Zhang-2019}, Pucci-Xiang-Zhang \cite{Pucci-Xiang-Zhang-2015}, Xiang-Zhang-R\u{a}dulescu \cite{Xiang-Zhang-Radulescu-2016} and the references therein.

Problems as in \eqref{problem} involving a Kirchhoff term are said to be degenerate if $a_0=0$ and nondegenerate if $a_0>0$. It is worth noting that the degenerate case is rather interesting and is treated in well-known papers in the Kirchhoff theory. We do cover the degenerate case in our paper which has several applications in physics. For example, the transverse oscillations of a stretched string with nonlocal flexural rigidity depends continuously on the Sobolev deflection norm of $u$ via $\psi(\int_\Omega |\nabla u|^2\,\mathrm{d} x)$, that is, $\psi(0)=0$ is nothing less than the base tension of the string is zero.

In the context of double phase problems of Kirchhoff  type there exist only few works. The first paper has been published by Fiscella-Pinamonti \cite{Fiscella-Pinamonti-2023} who considered the problem
\begin{align}\label{problem2534}
	-m \left[\int_\Omega \left( \frac{|\nabla u|^p}{p} + \mu(x) \frac{|\nabla u|^q}{q}\right)\,\mathrm{d} x\right]\mathcal{G}(u) =f(x,u) \quad \text{in } \Omega,\quad u  = 0 \quad \text{on } \partial\Omega,
\end{align}
with a right-hand side that has subcritical growth satisfying the Ambrosetti-Rabi\-no\-witz condition. The existence of a nontrivial weak solution of \eqref{problem2534} has been shown by applying the mountain pass theorem. Recently, Arora-Fiscella-Mukherjee-Win\-kert \cite{ Arora-Fiscella-Mukherjee-Winkert-2023} considered singular Kirchhoff double phase problems given by
\begin{align}\label{kirchhoff-critical}
	-m \left[\int_\Omega \left( \frac{|\nabla u|^p}{p} + \mu(x) \frac{|\nabla u|^q}{q}\right)\,\mathrm{d} x\right]\mathcal{G}(u) = \lambda u^{-\gamma} +u^{r-1} \quad \text{in } \Omega,
	\quad u\mid_{\partial\Omega} = 0
\end{align}
and used an appropriate decomposition of the Nehari manifold in order to prove the existence of two positive solutions with different energy sign, see also \cite{Arora-Fiscella-Mukherjee-Winkert-2022} by the same authors for the critical case to problem \eqref{kirchhoff-critical}. We also refer to the works of Cen-Vetro-Zeng \cite{Cen-Vetro-Zeng-2023} for degenerate double phase problems, Fiscella-Marino-Pinamonti-Verzellesi \cite{Fiscella-Marino-Pinamonti-Verzellesi-2023} for Kirchhoff double phase problems with Neumann boundary condition,  Gupta-Dwivedi \cite{Gupta-Dwivedi-2023} for mountain pass type solutions and  Ho-Winkert \cite{Ho-Winkert-2023} for infinitely many solutions via an abstract critical point result. It should be noted that all these works apply different methods than in our paper.

Let us now state the precise hypotheses to problem \eqref{problem} and the main results. We suppose the following assumptions:

\begin{enumerate}[label=\textnormal{(H$_1$)},ref=\textnormal{H$_1$}]
	\item\label{H1}
		$1<p<N$, $p<q<p^*=\frac{Np}{N-p}$ and $0 \leq \mu(\cdot)\in L^\infty(\Omega)$.
\end{enumerate}

\begin{enumerate}[label=\textnormal{(H$_2$)},ref=\textnormal{H$_2$}]
	\item\label{H2}
		$\psi\colon  [0,\infty) \to [0,\infty)$ is a continuous function defined by
		\begin{align*}
			\psi(s) = a_0 + b_0 s^{\vartheta-1} \quad \text{for all } s\geq 0,
		\end{align*}
		where $a_0 \geq 0$, $b_0>0$ and $\vartheta \geq 1$ such that $q\vartheta<p^*$.
\end{enumerate}

\begin{enumerate}[label=\textnormal{(H$_3$)},ref=\textnormal{H$_3$}]
	\item\label{H3}
		$f\colon \Omega \times \R \to \R$ is a Carath\'{e}odory function such that the following holds:
		\begin{enumerate}[label=\textnormal{(\roman*)},ref=\textnormal{\roman*}]
			\item\label{H3i}
				there exists a constant $c>0$ such that
				\begin{align*}
					|f(x,s)| & \leq c\left( 1 +|s|^{r-1}\right)
				\end{align*}
				for a.a.\,$x\in\Omega$, for all $s\in \R$, where $r<p^*$;
			\item\label{H3ii}
				\begin{align*}
					 & \lim_{s \to \pm \infty}\,\frac{f(x,s)}{|s|^{q\vartheta-2}s}=+\infty \quad\text{uniformly for a.a.\,}x\in\Omega;
				\end{align*}
			\item\label{H3iii}
				if $a_0 > 0$, assume
				\begin{align*}
					 & \lim_{s \to 0}\,\frac{f(x,s)}{|s|^{p-2}s}=0 \quad\text{uniformly for a.a.\,}x\in\Omega,
				\end{align*}
				and if $a_0 = 0$, assume
				\begin{align*}
					& \lim_{s \to 0}\,\frac{f(x,s)}{|s|^{p\vartheta-2}s}=0 \quad\text{uniformly for a.a.\,}x\in\Omega;
				\end{align*}
			\item\label{H3iv}
				the function
				\begin{align*}
					s\mapsto f(x,s)s- q\vartheta F(x,s)
				\end{align*}
				is nondecreasing on $[0,\infty)$ and nonincreasing on $(-\infty,0]$ for a.a.\,$x\in\Omega$, where $F(x,s)=\int_0^s f(x,t)\,\mathrm{d} t$;
			\item\label{H3v}
				the function
				\begin{align*}
					s\mapsto \frac{f(x,s)}{|s|^{q\vartheta-1}}
				\end{align*}
				is strictly increasing on $(-\infty,0)$ and on $(0,+\infty)$ for a.a.\,$x\in\Omega$.
		\end{enumerate}
\end{enumerate}

\begin{remark}
	\label{remark-H2}
	Note that \eqref{H3}\eqref{H3i} and \eqref{H3ii} imply $q\vartheta < r$. Also, from \eqref{H3}\eqref{H3i} and \eqref{H3}\eqref{H3ii} it follows that
	\begin{align}
		\label{superlinear-F}
		\lim_{s \to \pm \infty}\,\frac{F(x,s)}{|s|^{q\vartheta}}=+\infty \quad\text{uniformly for a.a.\,}x\in\Omega.
	\end{align}
\end{remark}

Now we can state the definition of a weak solution of problem \eqref{problem}.

\begin{definition}
	A function $u \in  W^{1, \mathcal{H} }_0 ( \Omega )$ is said to be a weak solution of problem \eqref{problem} if
	\begin{align*}
		\psi(\Phi_\mathcal{H}(\nabla u))
		\int_\Omega \left(|\nabla u|^{p-2}\nabla u + \mu(x) |\nabla u|^{q-2}\nabla u \right)\cdot \nabla v\,\mathrm{d} x= \int_{\Omega} f(x,u)v\,\mathrm{d} x
	\end{align*}
	is satisfied for all $v \in  W^{1, \mathcal{H} }_0 ( \Omega )$, where
	\begin{align*}
		\Phi_{\mathcal{H}}(\nabla u)= \int_\Omega \left(\frac{|\nabla u|^p}{p}+\mu(x) \frac{|\nabla u|^q}{q}\right)\,\mathrm{d} x=\frac{1}{p}\|\nabla u\|_p^p+\frac{1}{q}\|\nabla u\|_{q,\mu}^q.
	\end{align*}
\end{definition}

Our main result in this paper reads as follows.

\begin{theorem}
	\label{theorem_main_result}
	Let hypotheses \eqref{H1}--\eqref{H3} be satisfied. Then problem \eqref{problem} has at least two nontrivial constant sign solutions $u_0, v_0\in  W^{1, \mathcal{H} }_0 ( \Omega )$ such that
	\begin{align*}
		u_0(x) \geq 0 \quad\text{and}\quad v_0(x) \leq 0 \quad \text{for a.a.\,} x\in \Omega.
	\end{align*}
	In addition, a third sign-changing solution $y_0\in W^{1, \mathcal{H} }_0 ( \Omega )$ exists which turns out to be a least energy sign-changing solution for \eqref{problem}.
\end{theorem}

The proof of Theorem \ref{theorem_main_result} is based on a new technique that minimizes the corresponding energy functional of \eqref{problem} on a constraint set in which all sign-changing solutions of \eqref{problem} are contained. This set is different from the usual nodal Nehari manifold due to the nonlocal character of the problem. In order to demonstrate the different situation to other problems of similar type but without Kirchhoff term, let us consider the following problem
\begin{align}\label{problem-without-Kirchhoff}
	-\operatorname{div} \left(|\nabla u|^{p-2}\nabla u + \mu(x) |\nabla u|^{q-2}\nabla u \right)  = f(x,u)\quad  \text{in } \Omega, \quad
	u = 0 \quad  \text{on } \partial\Omega
\end{align}
with the same function $f\colon\Omega\times\R\to\R$. The energy functional $J\colon  W^{1, \mathcal{H} }_0 ( \Omega )\to\R$ of \eqref{problem-without-Kirchhoff} is given by
\begin{align*}
	J(u)=\frac{1}{p}\|\nabla u\|_p^p+\frac{1}{q}\|\nabla u\|_{q,\mu}^q-\int_\Omega F(x,u)\,\mathrm{d} x.
\end{align*}
Sign-changing solutions of \eqref{problem-without-Kirchhoff} can be found, for example, by minimizing the energy functional $J$ over the related nodal Nehari manifold
\begin{align*}
	\mathcal{N}_0= \left\{ u \in  W^{1, \mathcal{H} }_0 ( \Omega )\colon \pm u^\pm \in \mathcal{N} \right\},
\end{align*}
where
\begin{align*}
	\mathcal{N} = \left\{ u \in  W^{1, \mathcal{H} }_0 ( \Omega )\colon \langle J'(u) , u \rangle = 0, \; u \neq 0 \right\}
\end{align*}
is the Nehari manifold to \eqref{problem-without-Kirchhoff} with $\langle\cdot,\cdot\rangle$ being the duality pairing and $u^{\pm}$ denote the positive and negative part of $u$, respectively. Such treatment has been used for problems like \eqref{problem-without-Kirchhoff} in the papers by Crespo-Blanco-Winkert \cite{Crespo-Blanco-Winkert-2023} for variable exponent double phase problems, Liu-Dai \cite{Liu-Dai-2018} for superlinear problems, Gasi\'{n}ski-Papageorgiou \cite{Gasinski-Papageorgiou-2021} for locally Lipschitz right-hand sides and Gasi\'{n}ski-Winkert \cite{Gasinski-Winkert-2021} for Robin double phase problems. The successful  usage of this technique mainly relies on the following decompositions of the energy functional for suitable $u$
\begin{align*}
	\langle J'(u),u^+\rangle&=\langle J'(u^+),u^+\rangle,\quad
	\langle J'(u),-u^-\rangle=\langle J'(-u^-),-u^-\rangle, \\
	J(u)&=J(u^+)+J(-u^-).
\end{align*}
However, due to the appearance of the nonlocal term, such decompositions are not available anymore for problem \eqref{problem}. Indeed, let $\Psi\colon  [0,\infty) \to [0,\infty)$ be defined by
\begin{align*}
	\Psi(s)= \int_0^s\psi(t)\,\mathrm{d} t= a_0s  +\frac{b_0}{\vartheta}s^\vartheta,
\end{align*}
then the energy functional $\varphi\colon  W^{1, \mathcal{H} }_0 ( \Omega )\to\R$ of \eqref{problem} is given by
\begin{align*}
	\varphi(u) = \Psi[\Phi_{\mathcal{H}}(\nabla u)] -\int_\Omega F(x,u)\,\mathrm{d} x.
\end{align*}
For the functional $\varphi$ and $u\in W^{1,\mathcal{H}}_0(\Omega)$ with $u^+ \neq 0 \neq u^-$ we have the following relations when $\vartheta>1$
\begin{align*}
	\langle \varphi'(u),u^+\rangle&>\langle \varphi'(u^+),u^+\rangle,\quad
	\langle \varphi'(u),-u^-\rangle>\langle \varphi'(-u^-),-u^-\rangle, \\
	\varphi(u)&>\varphi(u^+)+\varphi(-u^-).
\end{align*}
Instead of using the set $\mathcal{N}_0$, we are going to minimize our functional on the constraint set
\begin{align}\label{constraint-set}
	\mathcal{M}=\Big\{u \in  W^{1, \mathcal{H} }_0 ( \Omega )\colon u^{\pm}\neq 0,\, \l\langle \varphi'(u),u^+ \r\rangle= \l\langle \varphi'(u),-u^- \r\rangle=0 \Big\}.
\end{align}
It is clear that all sign-changing solutions of \eqref{problem} belong to $\mathcal{M}$, so the global minimizer of $\varphi$ restricted to $\mathcal{M}$ is a least energy sign solution of \eqref{problem} provided it is a critical point. Note that the set $\mathcal{M}$ is not a manifold anymore. As far as we know, the set $\mathcal{M}$ was first introduced by Bartsch-Weth \cite{Bartsch-Weth-2005} in order to produce nodal solutions for the semilinear equation
\begin{align}\label{problem-Bartsch-Weth}
	-\Delta u +u=f(u) \quad\text{on }\Omega
\end{align}
with differentiable $f$ growing superlinearly and subcritically provided $\Omega$ contains a large ball. An extension of \eqref{problem-Bartsch-Weth} to the nonlocal case has been done by Shuai \cite{Shuai-2015}. The constraint set $\mathcal{M}$  was then used by several authors to different types of problems in order to get sign-changing solutions of the considered problems. The existence of ground state sign-changing solutions for Kirchhoff type problems in bounded domains of the shape
\begin{align*}
	-\left(a+b \|\nabla u\|_2^2\right)\Delta u = f(u)\quad  \text{in } \Omega, \quad
	u = 0 \quad  \text{on } \partial\Omega
\end{align*}
has been proved by Tang-Cheng \cite{Tang-Cheng-2016} under weaker assumptions on $f$ as in \cite{Shuai-2015} while Schr\"{o}dinger-Kirchhoff problems of the form
\begin{align}\label{problem-Schroedinger-Kirchhoff-type}
	-\left(a+b \|\nabla u\|_2^2\right)\Delta u+V(x)u = f(u)\quad  \text{in } \R^3, \quad
	u \in H^1(\R^3)
\end{align}
were treated by Tang-Chen \cite{Tang-Chen-2017} and Wang-Zhang-Cheng \cite{Wang-Zhang-Cheng-2018} under different assumptions on $f$ and $V$. Both papers present least energy sign-changing solutions for \eqref{problem-Schroedinger-Kirchhoff-type} using the constraint set $\mathcal{M}$. Related results using the same technique can be found in the papers by Liang-R\u{a}dulescu \cite{Liang-Radulescu-2020} for fractional Kirchhoff problems with logarithmic and critical nonlinearity, Zhang \cite{Zhang-2021} for Schr\"{o}dinger-Kirchhoff-type problems and Zhang \cite{Zhang-2023} for $N$-Laplacian equations of Kirchhoff type.

Recently, Chahma-Chen \cite{Chahma-Chen-2023} considered the quasilinear $p$-Laplacian Kirchhoff problem
\begin{align}\label{problem-p-Laplace}
	-\left(a+b \|\nabla u\|_p^p\right)\Delta_p u = \lambda f(x,u)+|u|^{p^*-2}u\quad  \text{in } \Omega, \quad
	u = 0 \quad  \text{on } \partial\Omega,
\end{align}
and obtained a least energy sign-changing solution of \eqref{problem-p-Laplace} which is strictly larger than twice of that of any ground state solution of \eqref{problem-p-Laplace}. The proofs are based on topological degree theory in combination with the quantitative deformation lemma. In this direction, we also mention the work of Li-Wang-Zhang \cite{Li-Wang-Zhang-2020} who studied the existence of ground state sign-changing solutions for $p$-Laplacian Kirchhoff-type problems with logarithmic nonlinearity given by
\begin{align*}
	-\left(a+b \|\nabla u\|_p^p\right)\Delta_p u = |u|^{q-2}u\ln u^2\quad  \text{in } \Omega, \quad
	u = 0 \quad  \text{on } \partial\Omega,
\end{align*}
where $a,b>0$, $4\leq 2p <q<p^*$ and $N>p$. A $(p,q)$-Laplacian Kirchhoff-type problems given by
\begin{equation}\label{Kirchhoff-pq}
	\begin{aligned}
		&-\left(a+b \|\nabla u\|_p^p\right)\Delta_p u-\left(c+d \|\nabla u\|_q^q\right)\Delta_q u\\
		&+V(x)\left(|u|^{p-2}u+|u|^{q-2}u\right)=K(x)f(u)\quad \text{in } \R^3
	\end{aligned}
\end{equation}
has been treated by Isernia-Repov\v{s} \cite{Isernia-Repovs-2021} who proved the existence of a least energy sign solution and, in addition, if $f$ is odd, infinitely many nontrivial solutions of \eqref{Kirchhoff-pq}. Finally, we also mention the papers of Papageorgiou-R\u{a}dulescu-Repov\v{s} \cite{Papageorgiou-Radulescu-Repovs-2019-b,Papageorgiou-Radulescu-Repovs-2020} for properties of the spectrum of the double phase operator and for ground state solutions of related equations involving this operator, respectively.

We point out that in all the above mentioned works the Kirchhoff function is of type $s\mapsto a+bs$ while in our case we consider more general functions given by $s\mapsto a_0+b_0 s^{\vartheta-1}$ with $\vartheta\geq 1$. So our results do not only extend the ones above to double phase operator but also to a more general Kirchhoff function. Moreover, we do not need any differentiability on the function $f$. In our proofs for a least energy sign changing solution to \eqref{problem} we combine in a suitable way variational tools together with the Poincar\'{e}-Miranda existence theorem and the quantitative deformation lemma. The existence of the constant sign solutions is mainly based on the mountain pass structure of the problem.

The paper is organized as follows. In Section \ref{Section_2} we recall the definition of Musielak-Orlicz Sobolev spaces and state properties of the double phase operator. In addition, we mention the Poincar\'{e}-Miranda existence theorem and the quantitative deformation lemma in their version that will be used later. Section \ref{Section_3} deals with the existence of the constant sign solutions which is based on the fact that the related truncated energy functionals satisfy the Cerami condition. Finally, in Section \ref{Section_4}, via various auxiliary propositions, we prove the existence of a sign-changing solution via minimizing the energy functional of \eqref{problem} over the constraint set $\mathcal{M}$ given in \eqref{constraint-set}.

\section{Preliminaries}\label{Section_2}

In this section we recall some preliminary facts and tools which are needed in the sequel. For this purpose, let $\Omega \subseteq \R^N$, $N \geq 2$, be a bounded domain with Lipschitz boundary $\partial \Omega$. We denote by $L^{r}(\Omega)$ and $L^r ( \Omega ; \R^N )$ the usual Lebesgue spaces equipped with the norm $\| \cdot \|_r$ for $1 \leq r \leq \infty$ and  $W^{1,r}_0(\Omega)$ stands for the Sobolev space with zero traces endowed with the equivalent norm $\| \nabla \cdot \|_r$.

Let $M ( \Omega )$ be the set of all measurable functions $u \colon \Omega \to \R$ and let $\mathcal{H} \colon \Omega \times [0,\infty) \to [0,\infty)$ be the nonlinear function defined by
\begin{align*}
	\mathcal{H} ( x , t )= t^p + \mu(x) t^q,
\end{align*}
where we suppose condition \eqref{H1}. The Musielak-Orlicz space $L^\mathcal{H} ( \Omega )$ is defined by
\begin{align*}
	L^\mathcal{H} ( \Omega ) = \left \{ u \in M(\Omega)\colon \rho_{\mathcal{H}} ( u ) < +\infty \right \}
\end{align*}
equipped with the Luxemburg norm
\begin{align*}
	\| u \|_{\mathcal{H}} = \inf \left \{ \tau >0 \colon \rho_{\mathcal{H}} \left ( \frac{u}{\tau} \right ) \leq 1  \right \},
\end{align*}
where the modular function $\rho_{\mathcal{H}}$ is given by
\begin{align}
	\label{modular}
	\rho_{\mathcal{H}}(u):=\int_\Omega \mathcal{H} ( x , |u| ) \,\mathrm{d} x=\int_\Omega \big( |u|^{p} + \mu(x) | u |^q \big ) \, \mathrm{d} x.
\end{align}
Furthermore, we define the seminormed space
\begin{align*}
	L^q_\mu ( \Omega ) = \left \{ u \in M ( \Omega ) \colon \int_\Omega \mu(x) | u |^q \, \mathrm{d} x < +\infty \right \},
\end{align*}
which is endowed with the seminorm
\begin{align*}
	\| u \|_{q,\mu} = \left ( \int_\Omega \mu(x) | u |^q \, \mathrm{d} x \right)^{ \frac{1}{q} }.
\end{align*}
In the same way we define $L^q_\mu(\Omega;\R^N)$.

The Musielak-Orlicz Sobolev space $W^{1 , \mathcal{H} } ( \Omega )$ is defined by
\begin{align*}
	W^{1,\mathcal{H}}(\Omega)= \left \{ u \in L^\mathcal{H}(\Omega) \colon | \nabla u | \in L^{\mathcal{H}} ( \Omega ) \right\}
\end{align*}
equipped with the norm
\begin{align*}
	\| u \|_{1 , \mathcal{H} } = \| \nabla u \|_{\mathcal{H}} + \| u \|_{\mathcal{H}},
\end{align*}
where $\|\nabla u\|_\mathcal{H}=\|\,|\nabla u|\,\|_{\mathcal{H}}$. The completion of $C^\infty_0(\Omega)$ in $W^{1,\mathcal{H}}(\Omega)$ is denoted by $W^{1,\mathcal{H}}_0(\Omega)$. We know that $L^{\mathcal{H}}(\Omega)$, $W^{1,\mathcal{H}}(\Omega)$ and $W^{1,\mathcal{H}}_0(\Omega)$ are reflexive Banach spaces and we can equip the space $W^{1,\mathcal{H}}_0(\Omega)$ with the equivalent norm
\begin{align*}
	\|u\|=\|\nabla u\|_{\mathcal{H}},
\end{align*}
see the paper of Crespo-Blanco-Gasi\'{n}ski-Harjulehto-Winkert \cite{Crespo-Blanco-Gasinski-Harjulehto-Winkert-2022}.

The following proposition shows the relation between the norm $\|\cdot\|_{\mathcal{H}}$ and the modular function $\rho_\mathcal{H}$, see Liu-Dai \cite{Liu-Dai-2018} for its proof.

\begin{proposition}
	\label{proposition_modular_properties}
	Let \eqref{H1} be satisfied, $\lambda>0$, $y\in L^{\mathcal{H}}(\Omega)$ and $\rho_{\mathcal{H}}$ be defined by \eqref{modular}. Then the following hold:
	\begin{enumerate}
		\item[\textnormal{(i)}]
			If $y\neq 0$, then $\|y\|_{\mathcal{H}}=\lambda$ if and only if $ \rho_{\mathcal{H}}(\frac{y}{\lambda})=1$;
		\item[\textnormal{(ii)}]
			$\|y\|_{\mathcal{H}}<1$ (resp.\,$>1$, $=1$) if and only if $ \rho_{\mathcal{H}}(y)<1$ (resp.\,$>1$, $=1$);
		\item[\textnormal{(iii)}]
			If $\|y\|_{\mathcal{H}}<1$, then $\|y\|_{\mathcal{H}}^q\leq \rho_{\mathcal{H}}(y)\leq\|y\|_{\mathcal{H}}^p$;
		\item[\textnormal{(iv)}]
			If $\|y\|_{\mathcal{H}}>1$, then $\|y\|_{\mathcal{H}}^p\leq \rho_{\mathcal{H}}(y)\leq\|y\|_{\mathcal{H}}^q$;
		\item[\textnormal{(v)}]
			$\|y\|_{\mathcal{H}}\to 0$ if and only if $ \rho_{\mathcal{H}}(y)\to 0$;
		\item[\textnormal{(vi)}]
			$\|y\|_{\mathcal{H}}\to +\infty$ if and only if $ \rho_{\mathcal{H}}(y)\to +\infty$.
	\end{enumerate}
\end{proposition}

We also mention some useful embeddings for the spaces $L^{\mathcal{H}}(\Omega)$ and $W^{1,\mathcal{H}}_0(\Omega)$, see Crespo-Blanco-Gasi\'{n}ski-Harjulehto-Winkert \cite{Crespo-Blanco-Gasinski-Harjulehto-Winkert-2022}.

\begin{proposition}
	\label{proposition_embeddings}
	Let \eqref{H1} be satisfied. Then the following embeddings hold:
	\begin{enumerate}
		\item[\textnormal{(i)}]
			$L^{\mathcal{H}}(\Omega) \hookrightarrow L^{r}(\Omega) $ and $ W^{1, \mathcal{H} }_0 ( \Omega )\hookrightarrow W^{1,r}_0(\Omega)$ are continuous for all $r\in [1,p]$;
		\item[\textnormal{(ii)}]
			$ W^{1, \mathcal{H} }_0 ( \Omega ) \hookrightarrow L^{r}(\Omega)$ is continuous for all $r \in [1,p^*]$ and compact for all $r \in [1,p^*)$;
		\item[\textnormal{(iii)}]
			$L^{\mathcal{H}}(\Omega) \hookrightarrow L^q_\mu(\Omega)$ is continuous;
		\item[\textnormal{(iv)}]
			$L^{q}(\Omega) \hookrightarrow L^{\mathcal{H}}(\Omega)$ is continuous.
	\end{enumerate}
\end{proposition}

Let $A\colon  W^{1, \mathcal{H} }_0 ( \Omega )\to  W^{1, \mathcal{H} }_0 ( \Omega )^*$ be the nonlinear map defined by
\begin{align}
	\label{operator_representation}
	\langle A(u),v\rangle :=\int_\Omega \big(|\nabla u|^{p-2}\nabla u+\mu(x)|\nabla u|^{q-2}\nabla u \big)\cdot\nabla v \,\mathrm{d} x
\end{align}
for all $u,v\in W^{1, \mathcal{H} }_0 ( \Omega )$ with $\langle\,\cdot\,,\,\cdot\,\rangle$ being the duality pairing between $ W^{1, \mathcal{H} }_0 ( \Omega )$ and its dual space $ W^{1, \mathcal{H} }_0 ( \Omega )^*$.  The proof of the following proposition can be found in Crespo-Blanco-Gasi\'{n}ski-Harjulehto-Winkert \cite{Crespo-Blanco-Gasinski-Harjulehto-Winkert-2022}.

\begin{proposition}
	\label{properties_operator_double_phase}
	Let hypotheses \eqref{H1}  be satisfied. Then, the operator $A$ defined in \eqref{operator_representation} is bounded, continuous, strictly monotone and of type \textnormal{(S$_+$)}, that is,
	\begin{align*}
		u_n\rightharpoonup u \quad \text{in } W^{1, \mathcal{H} }_0 ( \Omega ) \quad\text{and}\quad  \limsup_{n\to\infty}\,\langle Au_n,u_n-u\rangle\le 0,
	\end{align*}
	imply $u_n\to u$ in $ W^{1, \mathcal{H} }_0 ( \Omega )$.
\end{proposition}

We refer to the book of Harjulehto-H\"{a}st\"{o} \cite{Harjulehto-Hasto-2019} as well as the papers of Colasuonno-Squassina \cite{Colasuonno-Squassina-2016}, Crespo-Blanco-Gasi\'{n}ski-Harjulehto-Winkert \cite{Crespo-Blanco-Gasinski-Harjulehto-Winkert-2022}, Perera-Squassina \cite{Perera-Squassina-2018} and Liu-Dai \cite{Liu-Dai-2018} for more information about Musielak-Orlicz Sobolev spaces and properties of double phase operators.

For $s \in \R$, we set $s^{\pm}=\max\{\pm s,0\}$ and for $u \in  W^{1, \mathcal{H} }_0 ( \Omega )$ we define $u^{\pm}(\cdot)=u(\cdot)^{\pm}$. We have
\begin{align*}
	u^{\pm} \in  W^{1, \mathcal{H} }_0 ( \Omega ), \quad |u|=u^++u^-, \quad u=u^+-u^-,
\end{align*}
see Crespo-Blanco-Gasi\'{n}ski-Harjulehto-Winkert \cite{Crespo-Blanco-Gasinski-Harjulehto-Winkert-2022}.

Let $X$ be a Banach space and $X^*$ its topological dual space. Given $\varphi\in C^1(X)$, we define
\begin{align*}
	K_\varphi=\left\{u\in X\colon \varphi'(u)=0\right\}
\end{align*}
being the critical set of $\varphi$. Moreover, we say that the functional $\varphi$ satisfies the Cerami condition or \textnormal{C}-condition if for every sequence $\{u_n\}_{n \in \N} \subseteq X$ such that $\{ \varphi(u_n) \}_{n \in \N} \subseteq \R$ is bounded and it also fulfills
\begin{align*}
	( 1 + \| u_n \| )\varphi'(u_n) \to 0 \quad\text{as } n \to \infty,
\end{align*}
then it contains a strongly convergent subsequence. Furthermore, we say that it satisfies the Cerami condition at the level $c \in \R$ or the \textnormal{C$_c$}-condition if it holds for all the sequences such that $\varphi (u_n) \to c$ as $n \to \infty$ instead of for all the bounded sequences.

The following version of the mountain pass theorem is taken from the book by Papageorgiou-R\u{a}dulescu-Repov\v{s} \cite[Theorem 5.4.6]{Papageorgiou-Radulescu-Repovs-2019}.

\begin{theorem}[Mountain pass theorem] \label{Th:MPT}
	Let X be a Banach space and suppose $\varphi \in C^1(X)$, $u_0, u_1 \in X$ with $\| u_1 - u_0 \| > \delta > 0$,
	\begin{align*}
		\max\{\varphi(u_0), \varphi(u_1)\} & \leq \inf\{\varphi(u) \colon \| u - u_0 \| = \delta \} = m_\delta, \\
		c = \inf_{ \gamma \in \Gamma} \max_{ 0 \leq t \leq 1}
		\varphi(\gamma (t)) \text{ with } \Gamma & = \{\gamma \in C([0, 1], X) \colon \gamma(0) = u_0, \gamma(1) = u_1\}
	\end{align*}
	and $\varphi$ satisfies the \textnormal{C$_c$}-condition. Then $c \geq m_\delta$ and $c$ is a critical value of $\varphi$. Moreover, if $c = m_\delta$, then there exists $u \in \partial B_\delta (u_0)$ such that $\varphi'(u) = 0$.
\end{theorem}

The proof of the following quantitative deformation lemma can be found in the book by Willem \cite[Lemma 2.3]{Willem-1996}.

\begin{lemma}[Quantitative deformation lemma] \label{Le:DeformationLemma}
	Let $X$ be a Banach space, $\varphi \in C^1(X;\R)$, $\emptyset \neq S \subseteq X$, $c \in \R$, $\varepsilon,\delta > 0$ such that for all $u \in \varphi^{-1}([c - 2\varepsilon, c + 2\varepsilon]) \cap S_{2 \delta}$ there holds $\| \varphi'(u) \|_* \geq 8\varepsilon / \delta$, where $S_{r} = \{ u \in X \colon d(u,S) = \inf_{u_0 \in S} \| u - u_0 \| < r \}$ for any $r > 0$.
	Then there exists $\eta \in C([0, 1] \times X; X)$ such that
	\begin{enumerate} [label=(\roman*),font=\normalfont]
		\item[\textnormal{(i)}]
			$\eta (t, u) = u$, if $t = 0$ or if $u \notin \varphi^{-1}([c - 2\varepsilon, c + 2\varepsilon]) \cap S_{2 \delta}$;
		\item[\textnormal{(ii)}]
			$\varphi( \eta( 1, u ) ) \leq c - \varepsilon$ for all $u \in \varphi^{-1} ( ( - \infty, c + \varepsilon] ) \cap S $;
		\item[\textnormal{(iii)}]
			$\eta(t, \cdot )$ is an homeomorphism of $X$ for all $t \in [0,1]$;
		\item[\textnormal{(iv)}]
			$\| \eta(t, u) - u \| \leq \delta$ for all $u \in X$ and $t \in [0,1]$;
		\item[\textnormal{(v)}]
			$\varphi( \eta( \cdot , u))$ is decreasing for all $u \in X$;
		\item[\textnormal{(vi)}]
			$\varphi(\eta(t, u)) < c$ for all $u \in \varphi^{-1} ( ( - \infty, c] ) \cap S_\delta$ and $t \in (0, 1]$.
	\end{enumerate}
\end{lemma}

A $N$-dimensional version of the Bolzano intermediate value theorem can be stated by the following Poincar\'{e}-Miranda existence theorem, see the book by Dinca-Mawhin \cite[Corollary 2.2.15]{Dinca-Mawhin-2021}.

\begin{theorem}[Poincar\'{e}-Miranda existence theorem]
	\label{theorem-poincare-miranda}
	Let $P = [-t_1, t_1] \times \cdots \times [-t_N, t_N]$ with $t_i > 0$ for $i \in {1, \ldots, N}$ and $d \colon P \to \R^N$ be continuous. If for each $i \in \{1, \ldots, N\}$ one has
	\begin{align*}
		\begin{aligned}
			d_i (a) & \leq 0 \quad\text{when } a \in P \text{ and } a_i = -t_i, \\
			d_i (a) & \geq 0 \quad\text{when } a \in P \text{ and } a_i = t_i,
		\end{aligned}
	\end{align*}
	then $d$ has at least one zero point in $P$.
\end{theorem}

%

\section{Constant sign solutions}\label{Section_3}

In this section we are going to prove the existence of constant sign solutions for problem \eqref{problem}. To this end, let $\varphi\colon  W^{1, \mathcal{H} }_0 ( \Omega ) \to \R$ be the associated energy functional  to problem \eqref{problem} defined by
\begin{align*}
	\varphi(u) = \Psi[\Phi_{\mathcal{H}}(\nabla u)] -\int_\Omega F(x,u)\,\mathrm{d} x,
\end{align*}
where $\Psi\colon  [0,\infty) \to [0,\infty)$ is given by
\begin{align*}
	\Psi(s)= \int_0^s\psi(t)\,\mathrm{d} t= a_0s  +\frac{b_0}{\vartheta}s^\vartheta.
\end{align*}
Moreover, we consider the positive and negative truncations of the energy functional $\varphi$, that is, $\varphi_{\pm}\colon  W^{1, \mathcal{H} }_0 ( \Omega )  \to \R$ defined by
\begin{align*}
	\varphi_{\pm}(u)
	 & =\Psi[\Phi_{\mathcal{H}}(\nabla u)] -\int_\Omega F(x,\pm u^{\pm})\,\mathrm{d} x.
\end{align*}

First, we prove that these truncated functionals satisfy the Cerami condition.

\begin{proposition}
	\label{proposition_cerami}
	Let hypotheses \eqref{H1}--\eqref{H3} be satisfied. Then the functionals $\varphi_{\pm}$ fulfill the Cerami condition.
\end{proposition}

\begin{proof}
	We only show the proof of $\varphi_+$, the proof for $\varphi_-$ can be done similarly. To this end, let $\{u_n\}_{n \geq 1} \subseteq  W^{1, \mathcal{H} }_0 ( \Omega ) $ be a sequence such that
	\begin{align}
		\label{c1}
		\left|\varphi_+(u_n)\right| \leq M_1 \quad \text{for some }M_1>0 \text{ and for all }n\in\N
	\end{align}
	and
	\begin{align}
		\label{c2}
		\left(1+\|u_n\|\right)\varphi_+'(u_n)\to 0 \quad\text{in } W^{1, \mathcal{H} }_0 ( \Omega )^*.
	\end{align}
	From \eqref{c2} we have
	\begin{equation}
		\label{c3}
		\begin{aligned}
			&\Bigg| \psi(\Phi_\mathcal{H}(\nabla u_n))
			\int_\Omega \left(|\nabla u_n|^{p-2}\nabla u_n + \mu(x) |\nabla u_n|^{q-2}\nabla u_n \right)\cdot \nabla v\,\mathrm{d} x\\
			&-\int_\Omega f\l(x,u_n^+\r)v\,\mathrm{d} x \Bigg| \leq \frac{\varepsilon_n\|v\|}{1+\|u_n\|}
		\end{aligned}
	\end{equation}
	for all $v \in  W^{1, \mathcal{H} }_0 ( \Omega ) $ with $\varepsilon_n \to 0^+$. We choose $v=-u_n^- \in  W^{1, \mathcal{H} }_0 ( \Omega ) $ in \eqref{c3} to get
	\begin{align*}
		\psi(\Phi_\mathcal{H}(\nabla u_n)) \l( \l\|\nabla u_n^-\r\|_p^p+\l\|\nabla u_n^-\r\|_{q,\mu}^q \r)\leq \varepsilon_n\quad\text{for all }n\in\N
	\end{align*}
	as $f(\cdot,0)=0$ due to \eqref{H3}\eqref{H3iii}. Since
	\begin{equation*}
		\psi(\Phi_{\mathcal{H}} (\nabla u_n)) \geq \frac{b_0}{q^{\vartheta-1}} [ \rho_{\mathcal{H}} (\nabla u_n) ]^{\vartheta-1} \geq \frac{b_0}{q^{\vartheta-1}} [ \rho_{\mathcal{H}} (\nabla u_n^-) ]^{\vartheta-1},
	\end{equation*}
	see \eqref{H2}, we conclude that $\rho_\mathcal{H}(\nabla u_n^-)\to 0$ as $n \to \infty$. But this is equivalent to
	\begin{align*}
		\l\|u_n^-\r\| \to 0\quad\text{as }n\to \infty,
	\end{align*}
	by Proposition \ref{proposition_modular_properties}(v), that is,
	\begin{align}
		\label{c4}
		u_n^- \to 0 \quad\text{in } W^{1, \mathcal{H} }_0 ( \Omega ).
	\end{align}
	From \eqref{c1} we obtain
	\begin{equation}
		\label{c5-}
		\begin{aligned}
			&q\vartheta \Psi[\Phi_{\mathcal{H}}(\nabla u_n)] -\int_\Omega q\vartheta F\l(x, u_n^+\r)\,\mathrm{d} x
			\leq M_2 \quad\text{for all }n\in\N
		\end{aligned}
	\end{equation}
	for some $M_2>0$. Inequality \eqref{c5-}  can be rewritten as
	\begin{equation}
		\label{c5}
		\begin{aligned}
			&q\vartheta a_0\l(\frac{1}{p}\|\nabla u_n\|_{p}^p+\frac{1}{q}\|\nabla u_n\|_{q,\mu}^q\r)
			+q b_0 \l(\frac{1}{p}\|\nabla u_n\|_{p}^p+\frac{1}{q}\|\nabla u_n\|_{q,\mu}^q\r)^\vartheta\\
			&-\int_\Omega q\vartheta F\l(x, u_n^+\r)\,\mathrm{d} x
			\leq M_2 \quad\text{for all }n\in\N.
		\end{aligned}
	\end{equation}
	Next, we take $v=u_n \in  W^{1, \mathcal{H} }_0 ( \Omega ) $ in \eqref{c3} and use again the fact that $f(\cdot,0)=0$ due to \eqref{H3}\eqref{H3iii}, which gives
	\begin{align*}
		-\psi(\Phi_\mathcal{H}(\nabla u_n)) \l( \l\|\nabla u_n\r\|_p^p+\l\|\nabla u_n\r\|_{q,\mu}^q \r)+\int_\Omega f\l(x,u_n^+\r)u_n^+\,\mathrm{d} x\leq \varepsilon_n
	\end{align*}
	for all $n\in\N$, or equivalently
	\begin{equation}
		\label{c6-}
		\begin{aligned}
			& -a_0\l(\l\|\nabla u_n\r\|_p^p+\l\|\nabla u_n\r\|_{q,\mu}^q\r)\\
			&-b_0\l(\frac{1}{p}\l\|\nabla u_n\r\|_p^p+\frac{1}{q}\l\|\nabla u_n\r\|_{q,\mu}^q\r)^{\vartheta-1}
			\l( \l\|\nabla u_n\r\|_p^p+\l\|\nabla u_n\r\|_{q,\mu}^q \r)\\
			&+\int_\Omega f\l(x,u_n^+\r)u_n^+\,\mathrm{d} x
			\leq \varepsilon_n \quad\text{for all }n\in\N.
		\end{aligned}
	\end{equation}
	Now we add \eqref{c5} and \eqref{c6-} to get
	\begin{equation}
		\label{c7-}
		\begin{aligned}
			&a_0\left(\frac{q\vartheta}{p}-1\right)\l\|\nabla u_n\r\|_p^p +a_0\left(\vartheta-1\right)\l\| \nabla u_n\r\|_{q,\mu}^q\\
			&+b_0\l(\frac{1}{p}\l\|\nabla u_n\r\|_p^p+\frac{1}{q}\l\|\nabla u_n\r\|_{q,\mu}^q\r)^{\vartheta-1} \l[\l(\frac{q}{p}-1\r) \l\|\nabla u_n\r\|_p^p\r]\\
			&+\int_\Omega \left(f\l(x,u_n^+\r)u_n^+- q\vartheta F\l(x, u_n^+\r)\right)\,\mathrm{d} x \leq M_3 \quad\text{for all }n\in\N.
		\end{aligned}
	\end{equation}

	{\bf Claim: } The sequence $\{u_n^+\}_{n \geq 1} \subseteq  W^{1, \mathcal{H} }_0 ( \Omega ) $ is bounded.

	We argue indirectly and assume, by passing to a subsequence if necessary, that
	\begin{align}
		\label{c7c}
		\l\|u_n^+\r\| \to +\infty \quad\text{as }n\to +\infty.
	\end{align}
	Let $y_n=\frac{u_n^+}{\l\|u_n^+\r\|}$ for $n\in \N$. Then we have $\|y_n\|=1$ and $y_n \geq 0$ for all $n \in \N$. Hence,  we may suppose that
	\begin{align}
		\label{c8}
		y_n \rightharpoonup y \quad \text{in } W^{1, \mathcal{H} }_0 ( \Omega )
		\quad\text{and}\quad
		y_n \to y \quad\text{in }L^{r}(\Omega)
	\end{align}
	with $y\in  W^{1, \mathcal{H} }_0 ( \Omega )$, $y\geq 0$,  see Proposition \ref{proposition_embeddings}(iii).

	{\bf Case 1:} $y \neq 0$.

	Defining $\Omega_+=\left\{x\in \Omega \colon y(x)>0 \right\}$, we have $|\Omega_+|_N>0$ and due to \eqref{c8} it follows that
	\begin{align*}
		u_n^+(x) \to +\infty \quad \text{for a.a.\,}x\in\Omega_+.
	\end{align*}
	Then, from Fatou's Lemma, \eqref{c7c} and   \eqref{superlinear-F} we obtain
	\begin{align}
		\label{c10}
		\int_{\Omega_+} \frac{F(x,u_n^+)}{\l\|u_n^+\r\|^{q\vartheta}}\,\mathrm{d} x\to +\infty.
	\end{align}
	From \eqref{H3}\eqref{H3i} and \eqref{H3}\eqref{H3ii} we have
	\begin{align}
		\label{c11}
		F(x,s) \geq -M_4  \quad \text{for a.a.\,}x\in\Omega, \text{ for all }s\in\R,
	\end{align}
	and for some $M_4>0$. Applying \eqref{c11} gives
	\begin{align*}
		\int_{\Omega} \frac{F\l(x,u_n^+\r)}{\l\|u_n^+\r\|^{q\vartheta}}\,\mathrm{d} x
		 & = \int_{\Omega_+} \frac{F\l(x,u_n^+\r)}{\l\|u_n^+\r\|^{q\vartheta}}\,\mathrm{d} x+\int_{\Omega\setminus \Omega_+} \frac{F\l(x,u_n^+\r)}{\l\|u_n^+\r\|^{q\vartheta}}\,\mathrm{d} x \\
		 & \geq \int_{\Omega_+} \frac{F\l(x,u_n^+\r)}{\l\|u_n^+\r\|^{q\vartheta}}\,\mathrm{d} x - \frac{M_4}{\l\|u_n^+\r\|^{q\vartheta}}|\Omega|_N.
	\end{align*}
	Combining this with \eqref{c7c} and \eqref{c10}, one has
	\begin{align}
		\label{c15}
		\int_{\Omega} \frac{F\l(x,u_n^+\r)}{\l\|u_n^+\r\|^{q\vartheta}}\,\mathrm{d} x\to +\infty.
	\end{align}
	On the other hand, without loss of generality, consider $\| u_n^- \| \leq 1$ and $\| u_n^+ \| \geq 1$ for all $n \in \N$, from \eqref{c1} and \eqref{c4} it follows that
	\begin{align*}
		& \int_{\Omega} \frac{F\l(x,u_n^+\r)}{\l\|u_n^+\r\|^{q\vartheta}}\,\mathrm{d} x  \\
		& \leq \frac{a_0}{p}\frac{1}{\l\|u_n^+\r\|^{{q\vartheta}-p}}\|\nabla y_n\|_{p}^p+\frac{a_0}{q}\frac{1}{\l\|u_n^+\r\|^{q(\vartheta-1)}}\|\nabla y_n\|_{q,\mu}^q \\
		& \quad +\frac{b_0}{\vartheta} \l(\frac{1}{p}\frac{1}{\l\|u_n^+\r\|^{q-p}}\|\nabla y_n\|_{p}^p+\frac{1}{q}\|\nabla y_n\|_{q,\mu}^q+1\r)^\vartheta+M_5 \\
		& \leq a_0\rho_{\mathcal{H}}(\nabla y_n) +b_0 \l(\rho_{\mathcal{H}}(\nabla y_n)+1\r)^\vartheta+M_5
	\end{align*}
	for all $n \in \N$ and for some $M_5>0$.  Then, due to \eqref{c7c}, $\|y_n\|=1$ for all $n\in\N$ and Proposition \ref{proposition_modular_properties}(i), we obtain
	\begin{align*}
		\int_{\Omega} \frac{F\l(x,u_n^+\r)}{\l\|u_n^+\r\|^{q\vartheta}}\,\mathrm{d} x\leq M_6 \quad\text{for all }n \in\N,
	\end{align*}
	for some $M_6>0$, which contradicts \eqref{c15}.

	{\bf Case 2:} $y \equiv 0$.

	Let $\kappa \geq 1$ and define
	\begin{align*}
		v_n =(q\kappa)^{\frac{1}{q}}y_n \quad\text{for all }n \in \N.
	\end{align*}
	From \eqref{c8} we have
	\begin{align*}
		v_n \rightharpoonup 0\quad \text{in } W^{1, \mathcal{H} }_0 ( \Omega )
		\quad\text{and}\quad
		v_n \to 0 \quad\text{in }L^{r}(\Omega),
	\end{align*}
	which implies that
	\begin{align}
		\label{c17}
		\int_\Omega F(x,v_n)\,\mathrm{d} x \to 0.
	\end{align}

	Now let $t_n \in [0,1]$ be such that
	\begin{align}
		\label{c18}
		\varphi_+\l(t_n u_n^+\r) =\max \left\{ \varphi_+\l(tu_n^+\r)\colon t\in [0,1] \right\}.
	\end{align}
	Since $\l\|u_n^+\r\|\to +\infty$ we can find $n_0 \in \N$ such that
	\begin{align}
		\label{c19}
		0 < \frac{(q\kappa)^{\frac{1}{q}}}{\l\|u_n^+\r\|} \leq 1 \quad\text{for all }n \geq n_0.
	\end{align}
	From \eqref{c18}, \eqref{c19}, Proposition \ref{proposition_modular_properties}(ii) and \eqref{c17} we get
	\begin{align*}
		\varphi\l(t_nu_n^+\r)
		& \geq  \varphi(v_n) \\
		& =\Psi[\Phi_{\mathcal{H}}(v_n)]-\int_\Omega F(x,v_n)\,\mathrm{d} x  \\
		& =a_0\l(\frac{1}{p}\|\nabla v_n\|_p^p+\frac{1}{q}\|\nabla v_n\|_{q,\mu}^q\r)+\frac{b_0}{\vartheta}\l(\frac{1}{p}\|\nabla v_n\|_p^p+\frac{1}{q}\|\nabla v_n\|_{q,\mu}^q\r)^\vartheta \\
		& \quad -\int_\Omega F(x,v_n)\,\mathrm{d} x   \\
		& \geq \frac{b_0}{\vartheta} \left( \frac{1}{p}q^{\frac{p}{q}}\kappa^{\frac{p}{q}} \|\nabla y_n\|_p^p + \kappa \|\nabla y_n\|_{q,\mu}^q \right)^\vartheta -\int_\Omega F(x,v_n)\,\mathrm{d} x  \\
		& \geq \frac{b_0}{\vartheta} \left( \min\left\{\frac{1}{p}q^{\frac{p}{q}},1\right\}\kappa^{\frac{p}{q}} \rho_\mathcal{H}(\nabla y_n) \right) ^\vartheta -\int_\Omega F(x,v_n)\,\mathrm{d} x \\
		& \geq \frac{b_0}{\vartheta} \left( \min\left\{\frac{1}{p}q^{\frac{p}{q}},1\right\}\kappa^{\frac{p}{q}} \right) ^\vartheta -M_7 \quad \text{for all }n \geq n_1,
	\end{align*}
	for some $n_1 \geq n_0$ and $M_7>0$. Because $\kappa \geq 1$ is arbitrary chosen, we see that
	\begin{align}
		\label{c20}
		\varphi_+\l(t_nu_n^+\r) \to +\infty \quad\text{as }n \to\infty.
	\end{align}
	Taking \eqref{c1} into account gives
	\begin{align}
		\label{c22}
		\varphi_+(0)=0 \quad\text{and}\quad \varphi_+(u_n^+) \leq M_8\quad \text{for all }n \in \N
	\end{align}
	and for some $M_8>0$. From \eqref{c20} and \eqref{c22} there exists $n_2\in\N$ such that
	\begin{align}
		\label{c23}
		t_n \in (0,1) \quad\text{for all }n \geq n_2.
	\end{align}
	Therefore, by the chain rule, using \eqref{c23} and \eqref{c18}, one has
	\begin{align*}
		0=\frac{d}{dt}\varphi_+\l(tu_n^+\r)\Big|_{t=t_n} = \l\langle \varphi_+' \l(t_n u_n^+\r),u_n^+\r\rangle \quad\text{for all }n  \geq n_2,
	\end{align*}
	which can be written as
	\begin{equation}
		\label{c24}
		\begin{aligned}
			&\l(a_0+b_0\l(\frac{1}{p}\l\|\nabla \l(t_nu_n^+\r)\r\|_p^p   +\frac{1}{q}\l\|\nabla \l(t_nu_n^+\r)\r\|_{q,\mu}^q\r)^{\vartheta-1}\r)\\
			&\times \l(\l\|\nabla \l(t_nu_n^+\r)\r\|_p^p+\l\|\nabla \l(t_nu_n^+\r)\r\|_{q,\mu}^q\r)\\
			&=\int_\Omega  f\left(x,t_nu_n^+\right)t_nu_n^+\,\mathrm{d} x
		\end{aligned}
	\end{equation}
	for all $n \geq n_2$. Then, from \eqref{c24}, $t_n\in (0,1)$, hypothesis \eqref{H3}\eqref{H3iv} and \eqref{c7-} we conclude
	\begin{align*}
		&q\vartheta \varphi_+(t_nu_n^+)\\
		&=q\vartheta a_0\l(\frac{1}{p}\l\|\nabla \l(t_nu_n^+\r)\r\|_p^p+\frac{1}{q}\l\|\nabla \l(t_nu_n^+\r)\r\|_{q,\mu}^q\r)\\
		&\quad +q b_0 \l(\frac{1}{p}\l\|\nabla \l(t_nu_n^+\r)\r\|_p^p+\frac{1}{q}\l\|\nabla \l(t_nu_n^+\r)\r\|_{q,\mu}^q\r)^\vartheta\\
		&\quad -\int_\Omega q\vartheta F\l(x,t_nu_n^+\r)\,\mathrm{d} x\\
		&=a_0\left(\frac{q\vartheta}{p}-1\right)\l\|\nabla \l(t_nu_n^+\r)\r\|_p^p +a_0\left(\vartheta-1\right)\l\| \nabla \l(t_nu_n^+\r)\r\|_{q,\mu}^q\\
		&\quad+b_0\l(\frac{1}{p}\l\|\nabla \l(t_nu_n^+\r)\r\|_p^p+\frac{1}{q}\l\|\nabla \l(t_nu_n^+\r)\r\|_{q,\mu}^q\r)^{\vartheta-1} \l[\l(\frac{q}{p}-1\r) \l\|\nabla \l(t_nu_n^+\r)\r\|_p^p\r]\\
		&\quad+\int_\Omega \left( f\l(x,t_nu_n^+\r)t_nu_n^+-q\vartheta F\l(x,t_nu_n^+\r) \right)\,\mathrm{d} x \\
		&\leq a_0\left(\frac{q\vartheta}{p}-1\right)\l\|\nabla u_n\r\|_p^p +a_0\left(\vartheta-1\right)\l\| \nabla u_n\r\|_{q,\mu}^q\\
		&\quad+b_0\l(\frac{1}{p}\l\|\nabla u_n\r\|_p^p+\frac{1}{q}\l\|\nabla u_n\r\|_{q,\mu}^q\r)^{\vartheta-1} \l[\l(\frac{q}{p}-1\r) \l\|\nabla u_n\r\|_p^p\r]\\
		&\quad+\int_\Omega \left( f\l(x,u_n^+\r)u_n^+-q\vartheta F\l(x,u_n^+\r) \right)\,\mathrm{d} x \\
		&\leq M_3
	\end{align*}
	for all $ n \geq n_2$, which is a contradiction to \eqref{c20}. This proves the Claim.

	From the Claim and \eqref{c4} we know that the sequence $\{u_n\}_{n\geq 1}\subseteq  W^{1, \mathcal{H} }_0 ( \Omega ) $ is bounded. So we may suppose that
	\begin{align}
		\label{c27}
		u_n \rightharpoonup u \quad\text{in } W^{1, \mathcal{H} }_0 ( \Omega ) \quad\text{and}\quad
		u_n \to u \quad\text{in }L^{r}(\Omega).
	\end{align}
	Testing \eqref{c3} with $v=u_n-u\in  W^{1, \mathcal{H} }_0 ( \Omega )$ we get
	\begin{equation}
		\label{c3888}
		\begin{aligned}
			& \l(a_0+b_0\l(\frac{1}{p}\|\nabla u_n\|_p^p+\frac{1}{q}\|\nabla u_n\|_{q,\mu}^q\r)^{\vartheta-1}\r)\\
			&\times \int_\Omega \left(|\nabla u_n|^{p-2}\nabla u_n + \mu(x) |\nabla u_n|^{q-2}\nabla u_n \right)\cdot \nabla (u_n-u)\,\mathrm{d} x\\
			&-\int_\Omega  f\l(x,u_n^+\r)(u_n-u)\,\mathrm{d} x \leq \varepsilon_n\|u_n-u\|.
		\end{aligned}
	\end{equation}
	By passing to the $\limsup$ as $n\to\infty$ and using \eqref{c27} and \eqref{H3}\eqref{H3i} for the $f$ term, we get from \eqref{c3888}
	\begin{align*}
		\limsup_{n\to\infty}\, \frac{b_0}{q^{\theta - 1}} \rho_{\mathcal{H}}(\nabla u_n) \langle A(u_n),u_n-u\rangle \leq 0.
	\end{align*}
	Passing to further subsequence, it must hold either $\rho_{\mathcal{H}}(\nabla u_n) \to 0$, so $u=0$, or
	\begin{align*}
		\limsup_{n\to\infty}\, \langle A(u_n),u_n-u\rangle \leq 0.
	\end{align*}
	From the \textnormal{(S$_+$)}-property of $A$, see Proposition \ref{properties_operator_double_phase}, we conclude that $u_n\to u$ in $ W^{1, \mathcal{H} }_0 ( \Omega )$. Hence, $\varphi_+$ fulfills the Cerami condition.
\end{proof}

The following proposition will be useful for later considerations.

\begin{proposition}
	\label{proposition_auxiliary_result}
	Let hypotheses \eqref{H1}--\eqref{H3} be satisfied. If $a_0=0$, then there exist constants $\hat{C}, \tilde{C}>0$ such that
	\begin{align*}
		\varphi(u),\, \varphi_{\pm}(u) \geq
		\begin{cases}
			\hat{C} \|u\|^{q\vartheta} - \tilde{C} \|u\|^{r}  & \text{if } \|u\| \leq 1, \\
			\hat{C} \|u\|^{p\vartheta} -  \tilde{C} \|u\|^{r} & \text{if } \|u\|> 1.
		\end{cases}
	\end{align*}
	If $a_0>0$, then $\hat{C}, \tilde{C}>0$ satisfy
	\begin{align*}
		\varphi(u),\, \varphi_{\pm}(u) \geq
		\begin{cases}
			\hat{C} \|u\|^q - \tilde{C} \|u\|^{r}  & \text{if } \|u\| \leq 1, \\
			\hat{C} \|u\|^p -  \tilde{C} \|u\|^{r} & \text{if } \|u\|> 1.
		\end{cases}
	\end{align*}
\end{proposition}

\begin{proof}
	We are only going to prove the assertion for the functional $\varphi$, the proofs for $\varphi_{\pm}$ can be done similarly.

	Assume that $a_0=0$. From \eqref{H3}\eqref{H3i} and \eqref{H3}\eqref{H3iii}, for a given $\varepsilon>0$, we can find $\hat{c}=\hat{c}(\varepsilon)>0$ such that
	\begin{align}
		\label{l1}
		F(x,s) \leq \frac{\varepsilon}{p\vartheta}|s|^{p\vartheta} +\hat{c}|s|^{r} \quad \text{for a.a.\,}x\in \Omega.
	\end{align}
	Let $u \in  W^{1, \mathcal{H} }_0 ( \Omega )$. Applying \eqref{l1} and Proposition \ref{proposition_embeddings}(ii) leads to
	\begin{align*}
		\varphi(u)
		&=\Psi\l[\Phi_{\mathcal{H}}(u)\r]-\int_\Omega  F(x,u)\,\mathrm{d} x\\
		&\geq \frac{b_0}{\vartheta}\l(\frac{1}{p^\vartheta}\|\nabla u\|_p^{p\vartheta}+\frac{1}{q^\vartheta}\|\nabla u\|_{q,\mu}^{q\vartheta}\r)
		-\frac{\varepsilon}{p\vartheta} \|u\|_{p\vartheta}^{p\vartheta} - \hat{c} \|u\|_{r}^{r} \\
		&\geq \l(\frac{b_0}{p^{\vartheta}\vartheta}-\frac{C^{p\vartheta}_\Omega\varepsilon}{p\vartheta}\r)\|\nabla u\|_p^{p\vartheta}+\frac{b_0}{q^{\vartheta}\vartheta}\|\nabla u\|_{q,\mu}^{q\vartheta} -\hat{c} \l(C_\Omega^\mathcal{H}\r)^r\|u\|^r\\
		& \geq \min \l\{\l(\frac{b_0}{p^{\vartheta}\vartheta}-\frac{C^{p\vartheta}_\Omega\varepsilon}{p\vartheta}\r),\frac{b_0}{q^{\vartheta}\vartheta}\r\} \frac{1}{2^{\vartheta-1}} \left[ \rho_{\mathcal{H}}(\nabla u) \right]^\vartheta - \hat{c} \l(C_\Omega^{\mathcal{H}}\r)^{r} \|u\|^{r},
	\end{align*}
	with $C_\Omega$, $C_\Omega^{\mathcal{H}}$ being the embedding constants of the embeddings $W^{1,p}_0(\Omega)\to L^{p\vartheta}(\Omega)$ and $ W^{1, \mathcal{H} }_0 ( \Omega ) \to L^{r}(\Omega)$, respectively. Taking $\varepsilon \in (0, \frac{b_0}{p^{\vartheta-1}C_\Omega^{p\vartheta}})$ and using Proposition \ref{proposition_modular_properties}(iii), (iv) the assertion of the proposition follows.

	Assume now that $a_0>0$. As above, for each $\varepsilon>0$ there exists $\hat{c}=\hat{c}(\varepsilon)>0$ such that
	\begin{align*}
		F(x,s) \leq \frac{\varepsilon}{p}|s|^p +\hat{c}|s|^{r} \quad \text{for a.a.\,}x\in \Omega.
	\end{align*}
	and with very similar arguments we deduce
	\begin{align*}
		\varphi(u)
		&=\Psi\l[\Phi_{\mathcal{H}}(u)\r]-\int_\Omega  F(x,u)\,\mathrm{d} x\\
		&\geq a_0\l(\frac{1}{p}\|\nabla u\|_p^p+\frac{1}{q}\|\nabla u\|_{q,\mu}^q\r) -\frac{\varepsilon}{p} \|u\|_p^p - \hat{c} \|u\|_{r}^{r} \\
		&\geq \frac{1}{p}\l(a_0-C^p_\Omega\varepsilon\r)\|\nabla u\|_p^p+\frac{a_0}{q}\|\nabla u\|_{q,\mu}^q -\hat{c} \l(C_\Omega^\mathcal{H}\r)^r\|u\|^r\\
		& \geq \min \l\{\frac{1}{p}\l(a_0-C^p_\Omega\varepsilon\r),\frac{a_0}{q}\r\} \rho_{\mathcal{H}}(\nabla u)- \hat{c} \l(C_\Omega^{\mathcal{H}}\r)^{r} \|u\|^{r},
	\end{align*}
	from where the proposition follows with the same arguments.
\end{proof}

The next result shows that $u=0$ is a local minimizer of the functionals $\varphi_{\pm}$.

\begin{proposition}
	\label{proposition_local_minimizers}
	Let hypotheses \eqref{H1}--\eqref{H3} be satisfied. Then $u=0$ is a strict local minimizer for the functionals $\varphi_{\pm}$.
\end{proposition}

\begin{proof}
	Let $u \in  W^{1, \mathcal{H} }_0 ( \Omega )$ be such that $\|u\|<1$. From Proposition \ref{proposition_auxiliary_result} we get
	\begin{align*}
		\varphi_+(u) \geq
			\hat{C} \|u\|^{q\vartheta} - \tilde{C} \|u\|^{r}
	\end{align*}
	Because of $q\vartheta<r$, see Remark \ref{remark-H2}, we can find $\eta \in (0,1)$ small enough such that
	\begin{align*}
		\varphi_+(u) >0=\varphi_+(0) \quad\text{for all } u \in  W^{1, \mathcal{H} }_0 ( \Omega ) \text{ with }0<\|u\|<\eta.
	\end{align*}
	Thus,  $u=0$ is a strict local minimizer of $\varphi_+$. Similarly, we show that $u=0$ is a strict local minimizer for the functional $\varphi_-$.
\end{proof}

Before we prove the existence of constant sign solutions to problem \eqref{problem}, we need the following proposition.

\begin{proposition}
	\label{proposition_unbounded_below}
	Let hypotheses \eqref{H1}--\eqref{H3} be satisfied. Then, for $u \in  W^{1, \mathcal{H} }_0 ( \Omega )$ with $u(x)\geq 0$ for a.a.\,$x\in\Omega$, it holds $\varphi_{\pm}(tu)\to -\infty$ as $t \to \pm\infty$.
\end{proposition}

\begin{proof}
	From hypotheses \eqref{H3}\eqref{H3i} and \eqref{H3}\eqref{H3ii}, for each $K>0$ there exists $C=C(K)>0$ such that
	\begin{align}
		\label{c4354}
		F(x,s) \geq \frac{K}{q\vartheta} |s|^{q\vartheta}-C
	\end{align}
	for a.a.\,$x\in\Omega$ and for all $s\in\R$. Let $u \in  W^{1, \mathcal{H} }_0 ( \Omega )$ with $u(x)\geq 0$ for a.a.\,$x\in\Omega$ and $t\in\R$ with $|t|\geq 1$. Applying \eqref{c4354} we obtain
	\begin{align*}
		\varphi_{\pm}(tu)
		 & \leq |t|^qa_0\l(\frac{1}{p}\|\nabla u\|_p^p+\frac{1}{q}\|\nabla u\|_{q,\mu}^q\r) \\
		 &\quad+|t|^{q\vartheta} \left[ \frac{b_0}{\vartheta}\l(\frac{1}{p}\|\nabla u\|_p^p+\frac{1}{q}\|\nabla u\|_{q,\mu}^q\r)^\vartheta -\frac{K}{q\vartheta}\|u\|_{q\vartheta}^{q\vartheta}\right]
		 +C|\Omega|_N.
	\end{align*}
	Since $\vartheta q\geq q > p$ by \eqref{H1} and \eqref{H2}, if we choose $K>0$ big enough, we observe that $\varphi_{\pm}(tu)\to -\infty$ as $t \to \pm\infty$.
\end{proof}

Now we are in the position to state and prove our main result about constant sign solutions for problem \eqref{problem}.

\begin{proposition}
	\label{proposition_constant_sign_solutions}
	Let hypotheses \eqref{H1}--\eqref{H3} be satisfied. Then problem \eqref{problem} has at least two nontrivial constant sign solutions $u_0, v_0\in  W^{1, \mathcal{H} }_0 ( \Omega )$ such that
	\begin{align*}
		u_0(x) \geq 0 \quad\text{and}\quad v_0(x) \leq 0 \quad \text{for a.a.\,} x\in \Omega.
	\end{align*}
\end{proposition}

\begin{proof}
	From Proposition \ref{proposition_cerami} we know that $\varphi_{\pm}$ satisfy the Cerami condition. This fact and Proposition \ref{proposition_local_minimizers} along with Theorem 5.7.6 in Papageorgiou-R\u{a}dulescu-Repov\v{s} \cite{Papageorgiou-Radulescu-Repovs-2019} allows us to find $\eta_{\pm} \in (0,1)$ small enough such that
	\begin{align}
		\label{const1}
		\varphi_{\pm}(0)=0<\inf\left\{\varphi_{\pm}(0)\colon \|u\|=\eta_{\pm}\right\}=m_{\pm}.
	\end{align}
	From \eqref{const1} and the Propositions \ref{proposition_cerami} as well as \ref{proposition_unbounded_below} we see that we can apply the mountain pass theorem (see Theorem \ref{Th:MPT}). This gives the existence of $u_0, v_0 \in  W^{1, \mathcal{H} }_0 ( \Omega )$ such that $u_0 \in K_{\varphi_+}, \ v_0 \in K_{\varphi_-}$ and
	\begin{align*}
		\varphi_+(0)=0<m_+ \leq \varphi_+(u_0)
		\quad \text{as well as}\quad
		\varphi_-(0)=0<m_-\leq \varphi_-(v_0).
	\end{align*}
	From the properties above it is easy to see that $u_0 \neq 0$ and $v_0\neq 0$. Furthermore, due to  $u_0 \in K_{\varphi_+}$, it holds $\varphi_+'(u_0)=0$, that is
	\begin{align*}
		 & \psi\l(\Psi_{\mathcal{H}}(\nabla u_0)\r)
		\int_\Omega \left(|\nabla u_0|^{p-2}\nabla u_0 + \mu(x) |\nabla u_0|^{q-2}\nabla u_0 \right)\cdot \nabla v\,\mathrm{d} x
		=\int_\Omega  f\l(x,u_0^+\r)v\,\mathrm{d} x
	\end{align*}
	for all $v \in  W^{1, \mathcal{H} }_0 ( \Omega )$. Taking $v=-u_0^-\in  W^{1, \mathcal{H} }_0 ( \Omega )$ and noticing that
	\begin{equation*}
		\psi(\Phi_{\mathcal{H}} (\nabla u_0)) \geq \frac{b_0}{q^{\vartheta-1}} [ \rho_{\mathcal{H}} (\nabla u_0) ]^{\vartheta-1} \geq \frac{b_0}{q^{\vartheta-1}} [ \rho_{\mathcal{H}} (\nabla u_0^-) ]^{\vartheta-1},
	\end{equation*}
	we obtain
	\begin{align*}
		\rho_{\mathcal{H}}(u_0^-)=0.
	\end{align*}
	Hence, by the definition of $\|\cdot\|$, it follows that $\|u_0^-\|=0$. So $u_0 \geq 0, u_0\neq 0$. Similarly, we show that $v_0 \leq 0, v_0 \neq 0$.
\end{proof}

\section{Least energy sign-changing solution}\label{Section_4}

This section is concerned with the existence of a least energy sign-changing solution of problem \eqref{problem}. For this purpose we introduce the following constraint set
\begin{align*}
	\mathcal{M}=\Big\{u \in  W^{1, \mathcal{H} }_0 ( \Omega )\colon u^{\pm}\neq 0,\, \l\langle \varphi'(u),u^+ \r\rangle= \l\langle \varphi'(u),-u^- \r\rangle=0 \Big\}.
\end{align*}

\begin{proposition}
	\label{proposition_nodal_unique-pair}
	Let hypotheses \eqref{H1}--\eqref{H3} be satisfied and let $u \in  W^{1, \mathcal{H} }_0 ( \Omega )$ with $u^{\pm}\neq 0$. Then there exists a unique pair of positive numbers $(\alpha_u,\beta_u)$ such that $\alpha_uu^+-\beta_uu^-\in \mathcal{M}$. Moreover, if $u\in \mathcal{M}$,
	\begin{align*}
		\varphi(t_1u^+-t_2u^-)\leq \varphi(u^+-u^-)=\varphi(u)
	\end{align*}
	for all $t_1,t_2\geq 0$ with strict inequality when $(t_1,t_2)\neq (1,1)$. Furthermore, for $u\in\mathcal{M}$ we have that
	\begin{enumerate}
		\item[\textnormal{(i)}]
			if $\alpha>1$ and $0 < \beta \leq \alpha$, then $ \langle \varphi' (\alpha u^+-\beta u^-),\alpha u^+\rangle < 0$;
		\item[\textnormal{(ii)}]
			if $\alpha<1$ and $0 < \alpha \leq \beta$, then $ \langle \varphi' (\alpha u^+-\beta u^-),\alpha u^+\rangle > 0$;
		\item[\textnormal{(iii)}]
			if $\beta>1$ and $0 < \alpha \leq \beta$, then $ \langle \varphi' (\alpha u^+-\beta u^-),-\beta u^-\rangle < 0$;
		\item[\textnormal{(iv)}]
			if $\beta<1$ and $0 < \beta \leq \alpha$, then $ \langle \varphi' (\alpha u^+-\beta u^-),-\beta u^-\rangle > 0$.
	\end{enumerate}
\end{proposition}

\begin{proof}
	We divide the proof into four steps.

	{\bf Step I:} We show the existence of $(\alpha_u,\beta_u) \in (0,\infty)\times (0,\infty)$ such that $\alpha_uu^+-\beta_uu^-\in \mathcal{M}$.

	By hypothesis \eqref{H3}\eqref{H3v} we have for $t\in (0,1)$ and for $|u(x)|>0$ a.e.\,$\in\Omega$ that
	\begin{align*}
		\frac{f(x,tu)(tu)}{t^{q\vartheta}|u|^{q\vartheta}} \leq \frac{f(x,u)u}{|u|^{q\vartheta}} \quad \text{for a.a.\,}x\in \Omega,
	\end{align*}
	which implies
	\begin{align}
		\label{n4}
		f(x,tu)u \leq t^{q\vartheta-1}f(x,u)u \quad \text{for a.a.\,}x\in \Omega.
	\end{align}
	Applying \eqref{n4} we have
	\begin{align*}
		& \l\langle \varphi'\l(\alpha u^+-\beta u^-\r),\alpha u^+\r\rangle  \\
		& =\l(a_0+b_0\l(\frac{1}{p}\l\|\nabla \l(\alpha u^+-\beta u^-\r) \r\|_p^p+\frac{1}{q}\l\|\nabla \l(\alpha u^+-\beta u^-\r) \r\|_{q,\mu}^q\r)^{\vartheta-1}\r) \\
		& \quad\times \l(\l\|\nabla \l(\alpha u^+\r) \r\|_p^p+\l\|\nabla \l(\alpha u^+\r) \r\|_{q,\mu}^q\r)  \\
		& \quad -\int_\Omega  f\l(x,\alpha u^+\r)\alpha u^+\,\mathrm{d} x  \\
		& \geq \frac{b_0}{p^{\vartheta-1}} \alpha^{p\vartheta} \l\|\nabla u^+\r\|_p^{p\vartheta}-\alpha^{q\vartheta}\int_\Omega  f\l(x, u^+\r) u^+\,\mathrm{d} x>0
	\end{align*}
	for $\alpha>0$ small enough and for all $\beta\geq 0$ due to $p<q$. Similarly, it holds
	\begin{align*}
		& \l\langle \varphi'\l(\alpha u^+-\beta u^-\r),-\beta u^-\r\rangle   \\
		& =\l(a_0+b_0\l(\frac{1}{p}\l\|\nabla \l(\alpha u^+-\beta u^-\r) \r\|_p^p+\frac{1}{q}\l\|\nabla \l(\alpha u^+-\beta u^-\r) \r\|_{q,\mu}^q\r)^{\vartheta-1}\r) \\
		& \quad\times \l(\l\|\nabla \l(\beta u^-\r) \r\|_p^p+\l\|\nabla \l(\beta u^-\r) \r\|_{q,\mu}^q\r) \\
		& \quad -\int_\Omega  f\l(x,-\beta u^-\r)\l(-\beta u^-\r)\,\mathrm{d} x  \\
		& \geq \frac{b_0}{p^{\vartheta-1}} \beta^{p\vartheta} \l\|\nabla u^-\r\|_p^{p\vartheta}-\beta^{q\vartheta}\int_\Omega  f\l(x, -u^-\r) \l(-u^-\r)\,\mathrm{d} x>0
	\end{align*}
	for $\beta>0$ small enough and for all $\alpha\geq 0$, again because of $p<q$. Hence, there exists a number $\eta_1>0$ such that
	\begin{align}
		\label{G1}
		 & \l\langle \varphi'\l(\eta_1 u^+-\beta u^-\r),\eta_1 u^+\r\rangle>0
		\quad\text{and}\quad
		\l\langle \varphi'\l(\alpha u^+-\eta_1 u^-\r),-\eta_1 u^-\r\rangle>0
	\end{align}
	for all $\alpha,\beta\geq 0$. Now we choose $\eta_2>\max\{1,\eta_1\}$. Then, for $\beta \in [0,\eta_2]$ we get
	\begin{align*}
		& \frac{\l\langle \varphi'\l(\eta_2 u^+-\beta u^-\r),\eta_2 u^+\r\rangle}{\eta_2^{q\vartheta}}   \\
		& =\frac{a_0 \l(\l\|\nabla \l(\eta_2 u^+\r) \r\|_p^p+\l\| \nabla \l(\eta_2 u^+\r) \r\|_{q,\mu}^q\r)}{\eta_2^{q\vartheta}}  \\
		 & \quad +\frac{b_0\l(\frac{1}{p}\l\|\nabla \l(\eta_2 u^+-\beta u^-\r) \r\|_p^p+\frac{1}{q}\l\|\nabla \l(\eta_2 u^+-\beta u^-\r) \r\|_{q,\mu}^q\r)^{\vartheta-1}
		}{\eta_2^{q\vartheta}}  \\
		& \quad\times \l(\l\|\nabla \l(\eta_2 u^+\r) \r\|_p^p+\l\| \nabla \l(\eta_2 u^+\r) \r\|_{q,\mu}^q\r)\\
		& \quad -\int_\Omega  \frac{f\l(x,\eta_2 u^+\r)\eta_2 u^+}{\eta_2^{q\vartheta}}\,\mathrm{d} x  \\
		& \leq a_0 \frac{1}{\eta_2^{q(\vartheta-1)}}\l(\l\|\nabla u^+ \r\|_p^p+\l\| \nabla  u^+ \r\|_{q,\mu}^q\r) \\
		& \quad +b_0
		\l(\l\|\nabla \l(u^+- u^-\r) \r\|_p^p+\l\|\nabla \l( u^+- u^-\r) \r\|_{q,\mu}^q\r)^{\vartheta}   \\
		& \quad -\int_\Omega  \frac{f\l(x,\eta_2 u^+\r)}{\l(\eta_2u^+\r)^{q\vartheta-1}}\l(u^+\r)^{q\vartheta}\,\mathrm{d} x<0
	\end{align*}
	for $\eta_2$ large enough due to hypothesis \eqref{H3}\eqref{H3ii}. In a similar way, for $\alpha \in [0,\eta_2]$, we show that
	\begin{align*}
		\frac{\l\langle \varphi'\l(\alpha u^+-\eta_2 u^-\r),-\eta_2 u^-\r\rangle}{\eta_2^{q\vartheta}}<0
	\end{align*}
	for $\eta_2$ large enough. In summary, we obtain
	\begin{align}
		\label{G2}
		\l\langle \varphi'\l( \eta_2 u^+-\beta u^-\r),\eta_2 u^+\r\rangle<0
		\quad\text{and}\quad
		\l\langle \varphi'\l(\alpha u^+-\eta_2 u^-\r),-\eta_2 u^-\r\rangle<0
	\end{align}
	for $\eta_2$ large enough and $\alpha,\beta \in [0,\eta_2]$. Consider the mapping $\Lambda_u \colon [0,\infty) \times [0,\infty) \to \R^2$ given by
	\begin{equation*}
		\Lambda_u (\alpha,\beta)
		= \left( \l\langle \varphi'\l( \alpha u^+-\beta u^-\r), \alpha u^+\r\rangle ,
		\l\langle \varphi'\l(\alpha u^+-\beta u^-\r),-\beta u^-\r\rangle \right).
	\end{equation*}
	From \eqref{G1} and \eqref{G2} along with the  Poincar\'{e}-Miranda existence theorem given in Theorem \ref{theorem-poincare-miranda}, there exist a pair $(\alpha_u,\beta_u)\in [\eta_1,\eta_2] \times [\eta_1,\eta_2] \subseteq (0,\infty)\times (0,\infty)$ such that $\Lambda_u (\alpha_u,\beta_u) = (0,0)$, that is, $\alpha_uu^+-\beta_uu^-\in \mathcal{M}$.

	{\bf Step II:} We show now the cases regarding the sign of $ \langle \varphi' (\alpha u^+-\beta u^-),\alpha u^+\rangle$ and $ \langle \varphi' (\alpha u^+-\beta u^-),-\beta u^-\rangle$ for $u \in \mathcal{M}$.

	Let $u \in \mathcal{M}$, we have
	\begin{equation}
		\label{G3}
		\begin{aligned}
			0
			&=\l\langle \varphi'(u),u^+\r\rangle\\
			&=\l(a_0 +b_0\l(\frac{1}{p}\l\|\nabla u\r\|_p^p+\frac{1}{q}\l\|\nabla u\r\|_{q,\mu}^q\r)^{\vartheta-1}\r)
			\l(\l\|\nabla u^+\r\|_p^p+\l\|\nabla u^+\r\|_{q,\mu}^q\r)\\
			&\quad -\int_\Omega  f\l(x,u^+\r)u^+\,\mathrm{d} x
		\end{aligned}
	\end{equation}
	and
	\begin{equation}
		\label{G4}
		\begin{aligned}
			0
			&=\l\langle \varphi'(u),-u^-\r\rangle\\
			&=\l(a_0 +b_0\l(\frac{1}{p}\l\|\nabla u\r\|_p^p+\frac{1}{q}\l\|\nabla u\r\|_{q,\mu}^q\r)^{\vartheta-1}\r)
			\l(\l\|\nabla u^-\r\|_p^p+\l\|\nabla u^-\r\|_{q,\mu}^q\r)\\
			&\quad -\int_\Omega  f\l(x,-u^-\r)\l(-u^-\r)\,\mathrm{d} x.
		\end{aligned}
	\end{equation}
	The proof of all cases will go by contradiction. For case \textnormal{(i)}, let $\alpha>1$ and $0 < \beta \leq \alpha$ and assume that
	\begin{equation}
		\label{G5}
		\begin{aligned}
			0
			&\leq\l\langle \varphi'\l(\alpha u^+-\beta u^-\r),\alpha u^+\r\rangle\\
			&=\l(a_0 +b_0\l(\frac{1}{p}\l\|\nabla \l(\alpha u^+-\beta u^-\r)\r\|_p^p+\frac{1}{q}\l\|\nabla \l(\alpha u^+-\beta u^-\r)\r\|_{q,\mu}^q\r)^{\vartheta-1}\r)\\
			&\quad\times\l(\l\|\nabla \l(\alpha u^+\r)\r\|_p^p+\l\|\nabla \l(\alpha u^+\r)\r\|_{q,\mu}^q\r)\\
			&\quad -\int_\Omega  f\l(x,\alpha u^+\r)\alpha u^+\,\mathrm{d} x.
		\end{aligned}
	\end{equation}
	As $\alpha>1$, we get from \eqref{G5}
	\begin{equation}
		\label{G7}
		\begin{aligned}
			0
			&\leq \l(a_0\alpha^q +b_0\alpha^{q\vartheta}\l(\frac{1}{p}\l\|\nabla u\r\|_p^p+\frac{1}{q}\l\|\nabla u\r\|_{q,\mu}^q\r)^{\vartheta-1}\r)
			\l(\l\|\nabla u^+\r\|_p^p+\l\|\nabla u^+\r\|_{q,\mu}^q\r)\\
			&\quad -\int_\Omega  f\l(x,\alpha u^+\r)\l(\alpha u^+\r)\,\mathrm{d} x.
		\end{aligned}
	\end{equation}
	Dividing \eqref{G7} by $\alpha^{q\vartheta}$ and combining it with \eqref{G3} gives
	\begin{equation}
		\label{G8}
		\begin{aligned}
			&\int_\Omega  \l(\frac{f\l(x,\alpha u^+\r)}{\l(\alpha u^+\r)^{q\vartheta-1}}-\frac{f\l(x, u^+\r)}{\l( u^+\r)^{q\vartheta-1}}\r)\l(u^+\r)^{q\vartheta}\,\mathrm{d} x\\
			&\leq a_0\l(\frac{1}{\alpha^{q(\vartheta-1)}}-1\r)
			\l(\l\|\nabla u^+\r\|_p^p+\l\|\nabla u^+\r\|_{q,\mu}^q\r).
		\end{aligned}
	\end{equation}
	Because of hypothesis \eqref{H3}\eqref{H3v}, the left-hand side of \eqref{G8} is strictly positive, but the right-hand side of \eqref{G8} is nonpositive, a contradiction. For case \textnormal{(ii)}, let $\alpha<1$ and $0 < \alpha \leq \beta$ and assume that \eqref{G5} holds in the opposite direction. Arguing analogously, this yields \eqref{G7} in the opposite direction, so we obtain \eqref{G8} in the opposite direction. But this is again a contradiction, since the left-hand side is strictly negative and the right-hand side is nonnegative. For case \textnormal{(iii)}, let $\beta>1$ and $0 < \alpha \leq \beta$ and assume
	\begin{equation}
		\label{G6}
		\begin{aligned}
			0
			&\leq\l\langle \varphi'\l(\alpha u^+-\beta u^-\r),-\beta u^-\r\rangle\\
			&=\l(a_0 +b_0\l(\frac{1}{p}\l\|\nabla \l(\alpha u^+-\beta u^-\r)\r\|_p^p+\frac{1}{q}\l\|\nabla \l(\alpha u^+-\beta u^-\r)\r\|_{q,\mu}^q\r)^{\vartheta-1}\r)\\
			&\quad\times\l(\l\|\nabla \l(\beta u^-\r)\r\|_p^p+\l\|\nabla \l(\beta u^-\r)\r\|_{q,\mu}^q\r)\\
			&\quad -\int_\Omega  f\l(x,-\beta u^-\r)\l(-\beta u^-\r)\,\mathrm{d} x.
		\end{aligned}
	\end{equation}
	As $\beta>1$, we get from \eqref{G6}
	\begin{equation}
		\label{G9}
		\begin{aligned}
			0
			&\leq \l(a_0\beta^q +b_0\beta^{q\vartheta}\l(\frac{1}{p}\l\|\nabla u\r\|_p^p+\frac{1}{q}\l\|\nabla u\r\|_{q,\mu}^q\r)^{\vartheta-1}\r)
			\l(\l\|\nabla u^-\r\|_p^p+\l\|\nabla u^-\r\|_{q,\mu}^q\r)\\
			&\quad -\int_\Omega  f\l(x,-\beta u^-\r)\l(-\beta u^-\r)\,\mathrm{d} x.
		\end{aligned}
	\end{equation}
	Dividing \eqref{G9} by $\beta^{q\vartheta}$ and combining it with \eqref{G4} gives
	\begin{equation}
		\label{G10}
		\begin{aligned}
			&\int_\Omega  \l(\frac{f\l(x,-u^-\r)}{\l(u^-\r)^{q\vartheta-1}}-\frac{f\l(x,-\beta u^-\r)}{\l(\beta u^-\r)^{q\vartheta-1}}\r)\l(u^-\r)^{q\vartheta}\,\mathrm{d} x\\
			&\leq a_0\l(\frac{1}{\beta^{q(\vartheta-1)}}-1\r)
			\l(\l\|\nabla u^-\r\|_p^p+\l\|\nabla u^-\r\|_{q,\mu}^q\r).
		\end{aligned}
	\end{equation}
	Because of hypothesis \eqref{H3}\eqref{H3v}, the left-hand side of \eqref{G10} is strictly positive, but the right-hand side of \eqref{G10} is nonpositive, a contradiction. For case \textnormal{(iv)}, let $\beta<1$ and $0 < \beta \leq \alpha$ and assume that \eqref{G6} holds in the opposite direction. Arguing analogously, this yields \eqref{G9} in the opposite direction, so we obtain \eqref{G10} in the opposite direction. But this is again a contradiction, since the left-hand side is strictly negative and the right-hand side is nonnegative.

	{\bf Step III:} In this step we will show that the pair $(\alpha_u,\beta_u)$ obtained in Step I is unique.

	{\it Case (a):}  $u\in \mathcal{M}$.

	We are going to prove that $(\alpha_u,\beta_u)=(1,1)$ is the unique pair of numbers such that $\alpha_u u^+-\beta_u u^-\in \mathcal{M}$. To this end, let $(\alpha_0,\beta_0)\in (0,\infty)\times (0,\infty)$ be such that $\alpha_0 u^+-\beta_0 u^-\in \mathcal{M}$. If $0<\alpha_0\leq\beta_0$, the cases \textnormal{(ii)} and \textnormal{(iii)} from Step II imply $1 \leq \alpha_0\leq\beta_0 \leq 1$, i.e.\,$\alpha_0=\beta_0=1$. Alternatively, if $0<\beta_0\leq\alpha_0$, the cases \textnormal{(i)} and \textnormal{(iv)} from Step II imply $1 \leq \beta_0\leq\alpha_0 \leq 1$, i.e.\,$\alpha_0=\beta_0=1$.

	{\it Case (b):}  $u\not\in \mathcal{M}$.

	We suppose there exist $(\alpha_1,\beta_1)$, $(\alpha_2,\beta_2)$ such that
	\begin{align*}
		\tau_1:=\alpha_1 u^+-\beta_1 u^-\in \mathcal{M}
		\quad\text{and}\quad
		\tau_2:=\alpha_2 u^+-\beta_2 u^-\in \mathcal{M}.
	\end{align*}
	Then we obtain
	\begin{align}
		\label{G11}
		\tau_2=\l(\frac{\alpha_2}{\alpha_1}\r)\alpha_1u^+-\l(\frac{\beta_2}{\beta_1}\r)\beta_1u^-
		=\l(\frac{\alpha_2}{\alpha_1}\r)\tau_1^+-\l(\frac{\beta_2}{\beta_1}\r)\tau_1^- \in \mathcal{M}.
	\end{align}
	Since $\tau_1\in \mathcal{M}$, we know from Case (a) that the pair $(1,1)$ such that $1\cdot \tau_1^+-1\cdot \tau_1^-\in \mathcal{M}$ is unique. Therefore, combining this with \eqref{G11}, it follows
	\begin{align*}
		\frac{\alpha_2}{\alpha_1}=\frac{\beta_2}{\beta_1}=1,
	\end{align*}
	which implies $\alpha_1=\alpha_2$ and $\beta_1=\beta_2$.

	{\bf Step IV:} Let $\Upsilon_u\colon [0,\infty)\times [0,\infty)\to \R$ be given by
	\begin{align*}
		\Upsilon_u(\alpha,\beta)=\varphi\l(\alpha u^+-\beta u^-\r).
	\end{align*}
	We will prove that the unique pair $(\alpha_u,\beta_u)$ from Step I and III is the unique maximum point of $\Upsilon_u$ on $[0,\infty)\times [0,\infty)$.

	The idea is first to show that $\Upsilon_u$ has a maximum and then we show that it cannot be achieved at a boundary point of $[0,\infty)\times [0,\infty)$. Then the assertion follows.

	Let $\alpha, \beta \geq 1$ and assume, without any loss of generality, that $\alpha \geq \beta\geq 1$. Then we have
	\begin{equation}
		\label{G12}
		\begin{aligned}
			\frac{\Upsilon_u(\alpha,\beta)}{\alpha^{q\vartheta}}
			&=\frac{\varphi(\alpha u^+-\beta u^-)}{\alpha^{q\vartheta}}\\
			&=\frac{a_0\l(\frac{1}{p}\l\|\nabla \l(\alpha u^+-\beta u^-\r)\r\|_p^p+\frac{1}{q}\l\|\nabla \l(\alpha u^+-\beta u^-\r)\r\|_{q,\mu}^q\r)}{\alpha^{q\vartheta}}\\
			&\quad +\frac{\frac{b_0}{\vartheta}\l(\frac{1}{p}\l\|\nabla \l(\alpha u^+-\beta u^-\r)\r\|_p^p+\frac{1}{q}\l\|\nabla \l(\alpha u^+-\beta u^-\r)\r\|_{q,\mu}^q\r)^\vartheta}{\alpha^{q\vartheta}}\\
			&\quad -\int_\Omega  \frac{F\l(x,\alpha u^+-\beta u^-\r)}{\alpha^{q\vartheta}}\,\mathrm{d} x\\
			&\leq \frac{1}{\alpha^{q(\vartheta-1)}} a_0\l(\frac{1}{p}\l\|\nabla u\r\|_p^p+\frac{1}{q}\l\|\nabla u\r\|_{q,\mu}^q\r)\\
			&\quad+\frac{b_0}{\vartheta}\l(\frac{1}{p}\l\|\nabla u\r\|_p^p+\frac{1}{q}\l\|\nabla u\r\|_{q,\mu}^q\r)^\vartheta\\
			&\quad -\int_\Omega  \frac{F\l(x,\alpha u^+\r)}{\l(\alpha u^+\r)^{q\vartheta}}\l(u^+\r)^{q\vartheta}\,\mathrm{d} x-\int_\Omega  \frac{F\l(x,-\beta u^-\r)}{\l|-\beta u^-\r|^{q\vartheta}}\frac{\beta^{q\vartheta} \l(u^-\r)^{q\vartheta}}{\alpha^{q\vartheta}}\,\mathrm{d} x.
		\end{aligned}
	\end{equation}
	From \eqref{superlinear-F} we know that
	\begin{align*}
		\lim_{\alpha\to \infty}\l(-\int_\Omega  \frac{F\l(x,\alpha u^+\r)}{\l(\alpha u^+\r)^{q\vartheta}}\l(u^+\r)^{q\vartheta}\,\mathrm{d} x\r) = -\infty
	\end{align*}
	and
	\begin{align*}
		\limsup_{\substack{\beta\to \infty \\ \alpha\geq \beta}}\l(-\int_\Omega  \frac{F\l(x,-\beta u^-\r)}{\l|-\beta u^-\r|^{q\vartheta}}\frac{\beta^{q\vartheta} \l(u^-\r)^{q\vartheta}}{\alpha^{q\vartheta}}\,\mathrm{d} x\r)\leq 0.
	\end{align*}
	Using this with \eqref{G12}, we see that $\lim_{|(\alpha,\beta)|\to\infty}\Upsilon_u(\alpha,\beta)=-\infty$. Therefore, $\Upsilon_u$ has a maximum.

	Now we show that a maximum point of $\Upsilon_u$ cannot be achieved on the boundary of $[0,\infty)\times [0,\infty)$. Suppose via contradiction that $(0,\beta_0)$ is a maximum point of $\Upsilon_u$ with $\beta_0\geq 0$. Then we have for $\alpha> 0$
	\begin{align*}
		\Upsilon_u(\alpha,\beta_0)
		 & =\varphi(\alpha u^+-\beta_0 u^-)\\
		 & =a_0\l(\frac{1}{p}\l\|\nabla \l(\alpha u^+-\beta_0 u^-\r)\r\|_p^p+\frac{1}{q}\l\|\nabla \l(\alpha u^+-\beta_0 u^-\r)\r\|_{q,\mu}^q\r)\\
		 & \quad +\frac{b_0}{\vartheta}\l(\frac{1}{p}\l\|\nabla \l(\alpha u^+-\beta_0 u^-\r)\r\|_p^p+\frac{1}{q}\l\|\nabla \l(\alpha u^+-\beta_0 u^-\r)\r\|_{q,\mu}^q\r)^\vartheta \\
		 & \quad -\int_\Omega  F\l(x,\alpha u^+-\beta_0 u^-\r)\,\mathrm{d} x
	\end{align*}
	and
	\begin{align*}
		 & \frac{\partial \Upsilon_u(\alpha,\beta_0)}{\partial \alpha}  \\
		 & =a_0\alpha^{p-1} \l\|\nabla u^+\r\|_p^p+a_0\alpha^{q-1}\l\|\nabla u^+\r\|_{q,\mu}^q \\
		 & \quad +b_0\l(\frac{1}{p}\l\|\nabla \l(\alpha u^+-\beta_0 u^-\r)\r\|_p^p+\frac{1}{q}\l\|\nabla \l(\alpha u^+-\beta_0 u^-\r)\r\|_{q,\mu}^q\r)^{\vartheta-1} \\
		 & \quad\times\l(\alpha^{p-1} \l\|\nabla u^+\r\|_p^p+\alpha^{q-1}\l\|\nabla u^+\r\|_{q,\mu}^q\r) \\
		 & \quad -\int_\Omega  f\l(x,\alpha u^+\r)u^+\,\mathrm{d} x.
	\end{align*}
	In the case that $a_0>0$, we have
	\begin{align*}
		\frac{\partial \Upsilon_u(\alpha,\beta_0)}{\partial \alpha}
		\geq a_0\alpha^{p-1} \l\|\nabla u^+\r\|_p^p-\int_\Omega  f\l(x,\alpha u^+\r)u^+\,\mathrm{d} x.
	\end{align*}
	Dividing it by $\alpha^{p-1}$ we get
	\begin{align*}
		& \frac{1}{\alpha^{p-1}}\dfrac{\partial \Upsilon_u(\alpha,\beta_0)}{\partial \alpha} \geq a_0 \l\|\nabla u^+\r\|_p^p-\int_\Omega  \frac{f\l(x,\alpha u^+\r)}{\l(\alpha u^+\r)^{p-1}}\l(u^+\r)^p\,\mathrm{d} x,
	\end{align*}
	where the second term on the right-hand side goes to zero as $\alpha\to 0$ by hypothesis \eqref{H3}\eqref{H3iii}. Hence $\frac{\partial \Upsilon_u(\alpha,\beta_0)}{\partial \alpha}>0$ for $\alpha>0$ sufficiently small which implies that $\Upsilon_u$ is increasing with respect to $\alpha$. This is a contradiction. In the case that $a_0=0$, we have
	\begin{align*}
		\frac{\partial \Upsilon_u(\alpha,\beta_0)}{\partial \alpha}
		\geq \frac{b_0}{p^{\vartheta-1}}\alpha^{p\vartheta-1} \l\|\nabla u^+\r\|_p^{p\vartheta}-\int_\Omega  f\l(x,\alpha u^+\r)u^+\,\mathrm{d} x,
	\end{align*}
	Dividing it by $\alpha^{p\vartheta-1}$ we get
	\begin{align*}
		& \frac{1}{\alpha^{p\vartheta-1}}\dfrac{\partial \Upsilon_u(\alpha,\beta_0)}{\partial \alpha} \geq \frac{b_0}{p^{\vartheta-1}} \l\|\nabla u^+\r\|_p^{p\vartheta}-\int_\Omega  \frac{f\l(x,\alpha u^+\r)}{\l(\alpha u^+\r)^{p\vartheta-1}}\l(u^+\r)^{p\vartheta}\,\mathrm{d} x,
	\end{align*}
	where again the second term on the right-hand side goes to zero as $\alpha\to 0$ by hypothesis \eqref{H3}\eqref{H3iii}. So $\frac{\partial \Upsilon_u(\alpha,\beta_0)}{\partial \alpha}>0$ for $\alpha>0$ sufficiently small and therefore $\Upsilon_u$ is increasing with respect to $\alpha$. This is again a contradiction. In a similar way, we can show that $\Upsilon_u$ cannot achieve its global maximum at a point $(\alpha_0,0)$ with $\alpha_0\geq 0$.

	Altogether, the global maximum must be achieved in $(0,K)^2$ for some $K>0$. Thus, it is a critical point of $\Upsilon_u$ and by Steps I and III, the only critical point of $\Upsilon_u$ is $(\alpha_u,\beta_u)$.
\end{proof}

\begin{proposition}
	\label{proposition_nodal_unique-pair-less-one}
	Let hypotheses \eqref{H1}--\eqref{H3} be satisfied and let $u \in  W^{1, \mathcal{H} }_0 ( \Omega )$ with $u^{\pm}\neq 0$ such that $\langle \varphi'(u), u^+\rangle \leq 0$ and $\langle \varphi'(u),- u^-\rangle \leq 0$. Then the unique pair $(\alpha_u,\beta_u)$ obtained in Proposition \ref{proposition_nodal_unique-pair} satisfies $0<\alpha_u,\beta_u \leq 1$. Alternatively, if $\langle \varphi'(u), u^+\rangle \geq 0$ and $\langle \varphi'(u),- u^-\rangle \geq 0$, then $1\leq\alpha_u,\beta_u$.
\end{proposition}

\begin{proof}
	First, we suppose that $0<\beta_u\leq \alpha_u$. Since $\alpha_u u^+-\beta_uu^-\in \mathcal{M}$, we have $\l\langle\varphi'\l(\alpha_u u^+-\beta_u u^-\r),\alpha_u u^+ \r\rangle=0$, that is,
	\begin{equation}
		\label{G13}
		\begin{aligned}
			&\l(a_0+b_0\l(\frac{1}{p}\l\|\nabla \l(\alpha_u u^+-\beta_u u^-\r) \r\|_p^p+\frac{1}{q}\l\|\nabla \l(\alpha_u u^+-\beta_u u^-\r) \r\|_{q,\mu}^q\r)^{\vartheta-1}\r)\\
			&\times \l(\l\|\nabla \l(\alpha_u u^+\r) \r\|_p^p+\l\|\nabla \l(\alpha_u u^+\r) \r\|_{q,\mu}^q\r)\\
			&-\int_\Omega  f\l(x,\alpha_u u^+\r)\alpha_u u^+\,\mathrm{d} x=0.
		\end{aligned}
	\end{equation}
	From the assumptions, we have $\langle \varphi'(u),u^+\rangle \leq 0$ which reads as
	\begin{equation}
		\label{G14}
		\begin{aligned}
			&\l(a_0+b_0\l(\frac{1}{p}\l\|\nabla u \r\|_p^p+\frac{1}{q}\l\|\nabla u \r\|_{q,\mu}^q\r)^{\vartheta-1}\r)
			\l(\l\|\nabla u^+ \r\|_p^p+\l\|\nabla u^+\r\|_{q,\mu}^q\r)\\
			&-\int_\Omega  f\l(x,u^+\r)u^+\,\mathrm{d} x \leq 0.
		\end{aligned}
	\end{equation}
	Assuming $\alpha_u > 1$ and dividing \eqref{G13} by $\alpha_u^{q\vartheta}$ it follows
	\begin{equation}
		\label{G15}
		\begin{aligned}
			0
			&\leq \l(a_0\frac{1}{\alpha_u^{(\vartheta-1)q}}+b_0\l(\frac{1}{p}\l\|\nabla u\r\|_p^p+\frac{1}{q}\l\|\nabla u \r\|_{q,\mu}^q\r)^{\vartheta-1}\r)\\
			&\quad\times\l(\l\|\nabla u^+\r\|_p^p+\l\|\nabla u^+ \r\|_{q,\mu}^q\r)\\
			&\quad -\int_\Omega  \frac{f\l(x,\alpha_u u^+\r)}{\l(\alpha_u u^+\r)^{q\vartheta-1}} \l(u^+\r)^{q\vartheta}\,\mathrm{d} x.
		\end{aligned}
	\end{equation}
	Combining \eqref{G14} and \eqref{G15} gives
	\begin{equation}
		\label{TH34}
		\begin{aligned}
			&\int_\Omega  \l(\frac{f\l(x,\alpha_u u^+\r)}{\l(\alpha_u u^+\r)^{q\vartheta-1}} - \frac{f\l(x,u^+\r)}{\l(u^+\r)^{q\vartheta-1}} \r) \l(u^+\r)^{q\vartheta}\,\mathrm{d} x\\
			&\leq a_0\l(\frac{1}{\alpha_u^{(\vartheta-1)q}}-1\r)
			\l(\l\|\nabla u^+\r\|_p^p+\l\|\nabla u^+ \r\|_{q,\mu}^q\r).
		\end{aligned}
	\end{equation}
	Because of hypotheses \eqref{H3}\eqref{H3v} we see that this is a contradiction. Therefore, $0<\beta_u\leq \alpha_u\leq1$. If $0<\alpha_u\leq \beta_u$, we use instead $\l\langle\varphi'\l(\alpha_u u^+-\beta_u u^-\r),-\beta_u u^- \r\rangle=0$ and $\langle \varphi'(u),- u^-\rangle \leq 0$. Assuming $\beta_u > 1$ and arguing analogously to \eqref{G15} and \eqref{TH34}, we get
	\begin{equation*}
		\begin{aligned}
			&-\int_\Omega  \l(\frac{f\l(x,-\beta_u u^-\r)}{\l(\beta_u u^-\r)^{q\vartheta-1}} - \frac{f\l(x,-u^-\r)}{\l(u^-\r)^{q\vartheta-1}} \r) \l(u^-\r)^{q\vartheta}\,\mathrm{d} x\\
			&\leq a_0\l(\frac{1}{\beta_u^{(\vartheta-1)q}}-1\r)
			\l(\l\|\nabla u^-\r\|_p^p+\l\|\nabla u^- \r\|_{q,\mu}^q\r),
		\end{aligned}
	\end{equation*}
	which is again a contradiction and hence $0<\alpha_u\leq \beta_u \leq 1$. The case with $\langle \varphi'(u), u^+\rangle \geq 0$ and $\langle \varphi'(u),- u^-\rangle \geq 0$ works in the same way but doing the inequalities in the opposite direction.
\end{proof}

Let $m_0=\inf\limits_{\mathcal{M}}  \varphi$.

\begin{proposition}\label{proposition_positive_infimum_1}
	Let hypotheses \eqref{H1}--\eqref{H3} be satisfied. Then $m_0>0$ and so the infimum is finite.
\end{proposition}

\begin{proof}
	By Proposition \ref{proposition_auxiliary_result} and the fact that $q \leq q\vartheta < r$, we have
	\begin{align}
		\label{new-5465415}
		\varphi(u) \geq \hat{\xi}>0 \quad\text{for all }u \in  W^{1, \mathcal{H} }_0 ( \Omega ) \text{ with }\|u\|=\eta_0
	\end{align}
	for some $\eta_0 \in (0,1)$ small enough.

	Now let $u \in \mathcal{M}$ and take $\alpha_0, \beta_0>0$ such that $\|\alpha_0 u^+-\beta_0u^-\|=\eta_0$. Then, from \eqref{new-5465415} and Proposition \ref{proposition_nodal_unique-pair} we get
	\begin{align*}
		0<\hat{\xi} \leq \varphi(\alpha_0 u^+-\beta_0u^-) \leq \varphi(u).
	\end{align*}
	Since $u \in \mathcal{M}$ was arbitrary chosen, it follows that $m_0>0$.
\end{proof}

\begin{proposition}
	\label{proposition_nodal_positive-infimum}
	Let hypotheses \eqref{H1}--\eqref{H3} be satisfied. Then $\varphi_{|\mathcal{M}}$ is sequentially coercive, i.e.\,for any sequence $\{u_n\}_{n\in\N}\subseteq \mathcal{M}$ such that $\|u_n\|\to+\infty$ it holds that $\varphi(u_n) \to \infty$.
\end{proposition}

\begin{proof}
	Let $\{u_n\}_{n\in\N}\subseteq \mathcal{M}$ be a sequence such that $\|u_n\|\to+\infty$. Let $y_n=\frac{u_n}{\|u_n\|}$, then
	\begin{equation}
		\label{G16}
		\begin{aligned}
			 & y_n \rightharpoonup y \quad\text{in } W^{1, \mathcal{H} }_0 ( \Omega )
			\quad\text{and}\quad
			y_n\to y \quad\text{in }L^{r}(\Omega) \text{ and a.e.\,in }\Omega, \\
			 & y_n^{\pm} \rightharpoonup  y^{\pm} \quad\text{in } W^{1, \mathcal{H} }_0 ( \Omega )
			\quad\text{and}\quad
			y_n^{\pm}\to y^{\pm} \quad\text{in }L^{r}(\Omega)\text{ and a.e.\,in }\Omega
		\end{aligned}
	\end{equation}
	for some $y=y^+-y^-\in  W^{1, \mathcal{H} }_0 ( \Omega )$. Assume first that $y\neq 0$. From Proposition \ref{proposition_modular_properties}(iv) we have for $\|u_n\| > 1$
	\begin{equation}
		\label{n14}
		\begin{aligned}
			\varphi(u_n)
			&=a_0\l(\frac{1}{p}\l\|\nabla u_n\r\|_p^p+\frac{1}{q}\l\|\nabla u_n\r\|_{q,\mu}^q\r)\\
			&\quad +\frac{b_0}{\vartheta}\l(\frac{1}{p}\l\|\nabla u_n\r\|_p^p+\frac{1}{q}\l\|\nabla u_n\r\|_{q,\mu}^q\r)^\vartheta\\
			& \quad -\int_\Omega  F\l(x,u_n\r)\,\mathrm{d} x\\
			&\leq \frac{a_0}{p}\|u_n\|^q+\frac{b_0}{\vartheta p^\vartheta}\|u_n\|^{q\vartheta} -\int_\Omega  F(x,u_n)\,\mathrm{d} x.
		\end{aligned}
	\end{equation}
	Then, dividing \eqref{n14} by $\|u_n\|^{q\vartheta}$, passing to the limit as $n\to \infty$ and applying \eqref{superlinear-F} (see \eqref{c10}, \eqref{c11} and \eqref{c15}) it follows that $\frac{\varphi(u_n)}{\|u_n\|^{q\vartheta}}\to -\infty$, a contradiction to $\varphi(u_n)\geq m_0>0$ for all $n \in \N$, see Proposition \ref{proposition_positive_infimum_1}. Hence, we know that $y=0$, which implies $y^+=y^-=0$. Since $u_n \in \mathcal{M}$, by using Proposition \ref{proposition_nodal_unique-pair}, \eqref{G16} and Proposition \ref{proposition_modular_properties}(ii), we have for each pair $(t_1,t_2)\in (0,\infty)\times(0,\infty)$ with $0<t_1\leq t_2$ that
	\begin{align*}
		\varphi(u_n)
		 & \geq \varphi(t_1 y_n^+-t_2y_n^-)  \\
		 & =a_0\l(\frac{1}{p}\l\|\nabla \l(t_1y_n^+-t_2y_n^-\r)\r\|_p^p+\frac{1}{q}\l\|\nabla \l(t_1y_n^+-t_2y_n^-\r)\r\|_{q,\mu}^q\r) \\
		 & \quad +\frac{b_0}{\vartheta}\l(\frac{1}{p}\l\|\nabla \l(t_1y_n^+-t_2y_n^-\r)\r\|_p^p+\frac{1}{q}\l\|\nabla \l(t_1y_n^+-t_2y_n^-\r)\r\|_{q,\mu}^q\r)^\vartheta \\
		 & \quad -\int_\Omega  F\l(x,t_1y_n^+-t_2y_n^-\r)\,\mathrm{d} x\\
		 & \geq \frac{b_0}{q^\vartheta \vartheta}\min \l\{t_1^{p\vartheta},t_1^{q\vartheta}\r\} \left[ \rho_{\mathcal{H}}\l(\nabla y_n\r) \right] ^\vartheta
		-\int_\Omega  F\l(x,t_1y_n^+\r)\,\mathrm{d} x  \\
		 & \quad-\int_\Omega  F\l(x,-t_2y_n^-\r)\,\mathrm{d} x\to \frac{b_0}{q^\vartheta \vartheta}\min \l\{t_1^{p\vartheta},t_1^{q\vartheta}\r\},
	\end{align*}
	since $\rho_{\mathcal{H}}\l(\nabla y_n\r)=1$. Hence, for any given $K>0$ we take $t_1>0$ large enough and then for $n \geq n_0=n_0(t_1)$ we have that $\varphi(u_n)>K$.
\end{proof}

\begin{proposition}
	\label{proposition_bound_elements_nehari}
	Let hypotheses \eqref{H1}--\eqref{H3} be satisfied. Then there exists $M>0$ such that $\|u^{\pm}\|\geq M>0$ for all $u\in\mathcal{M}$.
\end{proposition}

\begin{proof}
	Let $u\in \mathcal{M}$ with $\|u^{\pm}\| <1$. We only show the statement for $u^+$, the other case works similarly.  By definition we have for $u\in \mathcal{M}$
	\begin{equation}
		\label{prop_bound_1}
		\begin{aligned}
			 & \left(a_0+b_0\left(\frac{1}{p}\left\|\nabla \left(u^+-u^-\right)\right\|_p^p+\frac{1}{q}\left\|\nabla \left(u^+-u^-\right)\right\|_{q,\mu}^q\right)^{\vartheta-1}\right) \\
			 & \quad\times\left(\left\|\nabla u^+ \right\|_p^p+\left\|\nabla u^+ \right\|_{q,\mu}^q\right) \\
			 & = \int_\Omega f(x,u^+)u^+ \,\mathrm{d} x.
		\end{aligned}
	\end{equation}
	First we do the case of $a_0 > 0$. From \eqref{H3}\eqref{H3i} and \eqref{H3}\eqref{H3iii}, for a given $\varepsilon>0$, we can find $C_\varepsilon>0$ such that
	\begin{align}
		\label{prop_bound_2}
		|f(x,s)| \leq \varepsilon |s|^{p-1} +C_\varepsilon|s|^{r-1}
	\end{align}
	for a.a.\,$x\in \Omega$ and for all $s\in \R$. Using \eqref{prop_bound_2} in \eqref{prop_bound_1} along with $W^{1,p}_0(\Omega)\hookrightarrow L^p(\Omega)$ and $W^{1,\mathcal{H}}_0(\Omega)\hookrightarrow L^r(\Omega)$ we obtain
	\begin{align*}
		a_0\left(\left\|\nabla u^+ \right\|_p^p+\left\|\nabla u^+ \right\|_{q,\mu}^q\right) \leq \varepsilon \left\|u^+\right\|_p^p+C_\varepsilon \left\|u^+\right\|_r^r
		\leq \varepsilon C_\Omega^p \left\|\nabla u^+\right\|_p^p+C_\varepsilon C_r^r\left\|u^+\right\|^r.
	\end{align*}
	Choosing $\varepsilon \in \left( 0,\frac{a_0}{C_{\Omega}^{p}}\right) $ and applying Proposition \ref{proposition_modular_properties}(iii) yields
	\begin{align*}
		M \left\|u^+\right\|^q
		\leq M \rho_{\mathcal{H}}(\nabla u^+)
		\leq \left\|u^+\right\|^r
	\end{align*}
	for some $M>0$. Since $q<r$, see Remark \ref{remark-H2}, the assertion of the proposition follows. Now we do the case $a_0=0$. From \eqref{H3}\eqref{H3i} and \eqref{H3}\eqref{H3iii}, for a given $\varepsilon>0$, we can find $C_\varepsilon>0$ such that
	\begin{align}
		\label{prop_bound_3}
		|f(x,s)| \leq \varepsilon |s|^{p\vartheta-1} +C_\varepsilon|s|^{r-1}
	\end{align}
	for a.a.\,$x\in \Omega$ and for all $s\in \R$. Using \eqref{prop_bound_3} in \eqref{prop_bound_1} along with $W^{1,p}_0(\Omega)\hookrightarrow L^{p\vartheta}(\Omega)$ and $W^{1,\mathcal{H}}_0(\Omega)\hookrightarrow L^r(\Omega)$ we obtain
	\begin{align*}
		\frac{b_0}{q^{\vartheta-1}}\left(\left\|\nabla u^+ \right\|_p^{p\vartheta}+\left\|\nabla u^+ \right\|_{q,\mu}^{q\vartheta}\right)
		&\leq \varepsilon \left\|u^+\right\|_{p\vartheta}^{p\vartheta}+C_\varepsilon \left\|u^+\right\|_r^r\\
		&\leq \varepsilon C_\Omega^{p\vartheta} \left\|\nabla u^+\right\|_p^{p\vartheta}+C_\varepsilon C_r^r\left\|u^+\right\|^r.
	\end{align*}
	Choosing $\varepsilon \in \left( 0,\frac{b_0}{q^{\vartheta-1}C_{\Omega}^{p\vartheta}}\right)$ and applying the inequality $2^{1-\vartheta}(s+t)^\vartheta \leq s^\vartheta + t^\vartheta$ for all $s,t \geq 0$ and Proposition \ref{proposition_modular_properties}(iii) yields
	\begin{align*}
		M \left\|u^+\right\|^{q\vartheta}
		\leq M \left[ \rho_{\mathcal{H}}(\nabla u^+) \right] ^\vartheta
		\leq \left\|u^+\right\|^r
	\end{align*}
	for some $M>0$. Since $q\vartheta<r$, see Remark \ref{remark-H2}, the assertion of the proposition follows.
\end{proof}

\begin{proposition}
	\label{proposition_infimum_achieved}
	Let hypotheses \eqref{H1}--\eqref{H3} be satisfied. Then there exists $y_0\in\mathcal{M}$ such that $\varphi(y_0)=m_0$.
\end{proposition}

\begin{proof}
	Let $\{y_n\}_{n\in\N}\subseteq \mathcal{M}$ be a minimizing sequence, that is,
	\begin{align*}
		\varphi(y_n) \searrow m_0.
	\end{align*}
	From Proposition \ref{proposition_nodal_positive-infimum} we know that $\{y_n\}_{n\in\N}$ is bounded in $W^{1,\mathcal{H}}_0(\Omega)$, so in particular $\{y_n^+\}_{n\in\N}, \{y_n^-\}_{n\in\N}$ are bounded in $W^{1,\mathcal{H}}_0(\Omega)$ (see Proposition \ref{proposition_modular_properties} \textnormal{(vi)}). Therefore, we can suppose, for subsequences if necessary, not relabeled, that
	\begin{equation}
		\label{prop_inf_achieved_1}
		\begin{aligned}
			y_n^+ \rightharpoonup y_0^+ \quad\text{in }W^{1,\mathcal{H}}_0(\Omega),  \quad y_0^+ \geq 0, \\
			y_n^- \rightharpoonup y_0^- \quad\text{in }W^{1,\mathcal{H}}_0(\Omega),  \quad y_0^- \geq 0, \\
			y_n^+\to y_0^+ \quad\text{in }L^r(\Omega) \text{ and a.e.\,in } \Omega, \\
			y_n^-\to y_0^- \quad\text{in }L^r(\Omega) \text{ and a.e.\,in } \Omega, \\
		\end{aligned}
	\end{equation}
	First, due to \eqref{prop_inf_achieved_1} and \eqref{H3}\eqref{H3i}, we have
	\begin{equation}
		\label{prop_inf_achieved_2}
		\begin{aligned}
			\int_\Omega f(x,\alpha y_n^+)\alpha y_n^+\,\mathrm{d} x
			 & \to \int_\Omega f(x,\alpha y_0^+)\alpha y_0^+\,\mathrm{d} x\quad\text{as }n\to\infty, \\
			\int_\Omega f(x,-\beta y_n^-)(-\beta y_n^-)\,\mathrm{d} x
			 & \to \int_\Omega f(x,-\beta y_0^-)(-\beta y_0^-)\,\mathrm{d} x\quad\text{as }n\to\infty
		\end{aligned}
	\end{equation}
	and also
	\begin{equation}
		\label{prop_inf_achieved_3}
		\begin{aligned}
			\int_\Omega F(x,\alpha y_n^+)\,\mathrm{d} x
			 & \to \int_\Omega F(x,\alpha y_0^+)\,\mathrm{d} x\quad\text{as }n\to\infty,  \\
			\int_\Omega F(x,-\beta y_n^-)\,\mathrm{d} x
			 & \to \int_\Omega F(x,-\beta y_0^-)\,\mathrm{d} x\quad\text{as }n\to\infty
		\end{aligned}
	\end{equation}
	for all $\alpha, \beta >0$.

	{\bf Claim:} $y_0^+\neq 0 \neq y_0^-$

	We argue indirectly and suppose that $y_0^+=0$. Since $y_n\in \mathcal{M}$, we get
	\begin{align*}
		0
		& =\l\langle \varphi'\l(y_n\r), y_n^+\r\rangle \\
		& =\l(a_0+b_0\l(\frac{1}{p}\l\|\nabla y_n \r\|_p^p+\frac{1}{q}\l\|\nabla y_n \r\|_{q,\mu}^q\r)^{\vartheta-1}\r) \\
		& \quad \times \l(\l\|\nabla y_n^+\r\|_p^p+\l\|\nabla y_n^+ \r\|_{q,\mu}^q\r) -\int_\Omega  f\l(x,y_n^+\r)y_n^+\,\mathrm{d} x  \\
		& \geq \frac{b_0}{q^{\vartheta-1}}\l(\l\|\nabla y_n^+\r\|_p^p+\l\|\nabla y_n^+ \r\|_{q,\mu}^q\r)^\vartheta-\int_\Omega  f\l(x, y_n^+\r) y_n^+\,\mathrm{d} x.
	\end{align*}
	Due to \eqref{prop_inf_achieved_2} we obtain from the inequality above
	\begin{align*}
		\frac{b_0}{q^{\vartheta-1}} \left[ \rho_\mathcal{H}(\nabla y_n^+) \right] ^\vartheta
		\leq \int_\Omega  f\l(x, y_n^+\r) y_n^+\,\mathrm{d} x \to 0 \quad\text{as }n\to \infty.
	\end{align*}
	Therefore, $\rho_\mathcal{H}(\nabla y_n^+)\to 0$ as $n\to\infty$ and so, by Proposition \ref{proposition_modular_properties}(v), it follows that $y_n^+\to 0$ in $W^{1,\mathcal{H}}(\Omega)$. But from Proposition \ref{proposition_bound_elements_nehari} we know that $\|y_n^+\|\geq M>0$, a contradiction. Similar arguments show that $y_0^-\neq 0$. This proves the Claim.

	From the Claim, by using Proposition \ref{proposition_nodal_unique-pair}, since $y_0^+,y_0^-\neq 0$, we can find unique $\alpha_{y_0}, \beta_{y_0}>0$ such that $\alpha_{y_0}y_0^+ -\beta_{y_0}y_0^-\in \mathcal{M}$. Furthermore, from \eqref{prop_inf_achieved_2} and the weak lower semicontinuity of $\|\cdot\|_p$ and $\|\cdot\|_{q,\mu}$, it follows
	\begin{align*}
		& \left\langle \varphi'(y_0),\pm y_0^{\pm} \right\rangle  \\
		& = \l(a_0+b_0\l(\frac{1}{p}\l\|\nabla y_0 \r\|_p^p+\frac{1}{q}\l\|\nabla y_0 \r\|_{q,\mu}^q\r)^{\vartheta-1}\r) \l(\l\|\nabla\left(\pm y_0^{\pm}\right) \r\|_p^p+\l\|\nabla \left(\pm y_0^{\pm}\right) \r\|_{q,\mu}^q\r) \\
		& \quad -\int_\Omega  f\l(x, \pm y_0^{\pm} \r) \left(\pm y_0^{\pm}\right)\,\mathrm{d} x  \\
		& \leq \liminf_{n\to \infty}\l(a_0+b_0\l(\frac{1}{p}\l\|\nabla y_n \r\|_p^p+\frac{1}{q}\l\|\nabla y_n \r\|_{q,\mu}^q\r)^{\vartheta-1}\r)  \\
		& \qquad\qquad\quad\times\l(\l\|\nabla\left(\pm y_n^{\pm}\right) \r\|_p^p+\l\|\nabla \left(\pm y_n^{\pm}\right) \r\|_{q,\mu}^q\r) \\
		& \quad -\lim_{n\to \infty}\int_\Omega  f\l(x, \pm y_n^{\pm} \r) \left(\pm y_n^{\pm}\right)\,\mathrm{d} x \\
		& =\liminf_{n\to \infty}\left\langle \varphi'(y_n),\pm y_n^{\pm} \right\rangle=0,
	\end{align*}
	since $y_n \in \mathcal{M}$. Then we can apply Proposition \ref{proposition_nodal_unique-pair-less-one} to conclude that $ \alpha_{y_0}, \beta_{y_0} \in (0,1]$. Using this and \ref{H3}\eqref{H3iv} leads to
	\begin{equation}
		\label{prop_inf_achieved_4}
		\begin{aligned}
			&\frac{1}{q\vartheta}f(x,\alpha_{y_0} y_0^+)\alpha_{y_0} y_0^+-F(x,\alpha_{y_0} y_0^+)\\ & \leq \frac{1}{q\vartheta}f(x,y_0^+) y_0^+-F(x, y_0^+),   \\
			&\frac{1}{q\vartheta}f(x,-\beta_{y_0} y_0^-)(-\beta_{y_0} y_0^-)-F(x,-\beta_{y_0} y_0^-)\\
			& \leq \frac{1}{q\vartheta}f(x,-y_0^-) (-y_0^-)-F(x, -y_0^-)
		\end{aligned}
	\end{equation}
	for a.a.\,$x\in\Omega$. Finally, from $\alpha_{y_0}y_0^+ -\beta_{y_0}y_0^-\in \mathcal{M}$, $ \alpha_{y_0}, \beta_{y_0} \in (0,1]$, \eqref{prop_inf_achieved_2}, \eqref{prop_inf_achieved_3}, \eqref{prop_inf_achieved_4} and $y_n\in\mathcal{M}$ we conclude that
	\begin{align*}
		m_0
		 & \leq \varphi\left(\alpha_{y_0}y_0^+-\beta_{y_0}y_0^-\right)
		-\frac{1}{q\vartheta} \left\langle \varphi'\left(\alpha_{y_0}y_0^+-\beta_{y_0}y_0^-\right),\alpha_{y_0}y_0^+-\beta_{y_0}y_0^- \right\rangle   \\
		 & =a_0\l(\frac{\alpha_{y_0}^p}{p}\l\|\nabla y_0^+\r\|_p^p+\frac{\beta_{y_0}^p}{p}\l\|\nabla y_0^-\r\|_p^p+\frac{\alpha_{y_0}^q}{q}\l\|\nabla y_0^+\r\|_{q,\mu}^q+\frac{\beta_{y_0}^q}{q}\l\|\nabla y_0^-\r\|_{q,\mu}^q\r)  \\
		 & \quad +\frac{b_0}{\vartheta}\l(\frac{\alpha_{y_0}^p}{p}\l\|\nabla y_0^+\r\|_p^p+\frac{\beta_{y_0}^p}{p}\l\|\nabla y_0^-\r\|_p^p+\frac{\alpha_{y_0}^q}{q}\l\|\nabla y_0^+\r\|_{q,\mu}^q+\frac{\beta_{y_0}^q}{q}\l\|\nabla y_0^-\r\|_{q,\mu}^q\r)^{\vartheta} \\
		 & \quad -\int_\Omega  F\l(x,\alpha_{y_0} u^+\r)\,\mathrm{d} x-\int_\Omega  F\l(x,-\beta_{y_0} u^-\r)\,\mathrm{d} x  \\
		 & \quad -a_0\left(\frac{\alpha_{y_0}^p}{q\vartheta}\l\|\nabla y_0^+\r\|_p^p+\frac{\beta_{y_0}^p}{q\vartheta}\l\|\nabla y_0^-\r\|_p^p+\frac{\alpha_{y_0}^q}{q\vartheta}\l\|\nabla y_0^+\r\|_{q,\mu}^q+\frac{\beta_{y_0}^q}{q\vartheta}\l\|\nabla y_0^-\r\|_{q,\mu}^q\right)   \\
		 & \quad -\frac{b_0}{q\vartheta}\l(\frac{\alpha_{y_0}^p}{p}\l\|\nabla y_0^+\r\|_p^p+\frac{\beta_{y_0}^p}{ p}\l\|\nabla y_0^-\r\|_p^p+\frac{\alpha_{y_0}^q}{q}\l\|\nabla y_0^+\r\|_{q,\mu}^q+\frac{\beta_{y_0}^q}{ q}\l\|\nabla y_0^-\r\|_{q,\mu}^q\r)^{\vartheta-1} \\
		 & \qquad\times\left(\alpha_{y_0}^p\l\|\nabla y_0^+\r\|_p^p+\beta_{y_0}^p\l\|\nabla y_0^-\r\|_p^p+\alpha_{y_0}^q\l\|\nabla y_0^+\r\|_{q,\mu}^q+\beta_{y_0}^q\l\|\nabla y_0^-\r\|_{q,\mu}^q\right)   \\
		 & \quad+\frac{1}{q\vartheta}\int_\Omega f\left(x,\alpha_{y_0} y_0^+\right)\alpha_{y_0} y_0^+\,\mathrm{d} x
		+\frac{1}{q\vartheta}\int_\Omega f\left(x,-\beta_{y_0} y_0^-\right)(-\beta_{y_0} y_0^-)\,\mathrm{d} x  \\
		 & =a_0 \left(\frac{1}{p}-\frac{1}{q\vartheta}\right) \left(\alpha_{y_0}^p\l\|\nabla y_0^+\r\|_p^p+\beta_{y_0}^p\l\|\nabla y_0^-\r\|_p^p \right) \\
		 & \quad +a_0 \left(\frac{1}{q}-\frac{1}{q\vartheta}\right) \left(\alpha_{y_0}^q\l\|\nabla y_0^+\r\|_{q,\mu}^q+\beta_{y_0}^q\l\|\nabla y_0^-\r\|_{q,\mu}^q \right) \\
		 & \quad +b_0\l(\frac{\alpha_{y_0}^p}{p}\l\|\nabla y_0^+\r\|_p^p+\frac{\beta_{y_0}^p}{ p}\l\|\nabla y_0^-\r\|_p^p+\frac{\alpha_{y_0}^q}{q}\l\|\nabla y_0^+\r\|_{q,\mu}^q+\frac{\beta_{y_0}^q}{ q}\l\|\nabla y_0^-\r\|_{q,\mu}^q\r)^{\vartheta-1}  \\
		 & \quad \times \Bigg [ \left(\frac{1}{p\vartheta}-\frac{1}{q\vartheta}\right) \left(\alpha_{y_0}^p\l\|\nabla y_0^+\r\|_p^p+\beta_{y_0}^p\l\|\nabla y_0^-\r\|_p^p \right)  \\
		 & \qquad \quad  +\left(\frac{1}{q\vartheta}-\frac{1}{q\vartheta}\right) \left(\alpha_{y_0}^q\l\|\nabla y_0^+\r\|_{q,\mu}^q+\beta_{y_0}^q\l\|\nabla y_0^-\r\|_{q,\mu}^q \right)\Bigg]  \\
		 & \quad +\int_\Omega \left(\frac{1}{q\vartheta}f(x,\alpha_{y_0} y_0^+)\alpha_{y_0} y_0^+-F(x,\alpha_{y_0} y_0^+)\right)\,\mathrm{d} x  \\
		 & \quad +\int_\Omega \left(\frac{1}{q\vartheta}f(x,-\beta_{y_0} y_0^-)(-\beta_{y_0} y_0^-)-F(x,-\beta_{y_0} y_0^-)\right)\,\mathrm{d} x   \\
		 & \leq a_0 \left(\frac{1}{p}-\frac{1}{q\vartheta}\right) \left(\l\|\nabla y_0^+\r\|_p^p+\l\|\nabla y_0^-\r\|_p^p \right)    \\
		 & \quad +a_0 \left(\frac{1}{q}-\frac{1}{q\vartheta}\right) \left(\l\|\nabla y_0^+\r\|_{q,\mu}^q+\l\|\nabla y_0^-\r\|_{q,\mu}^q \right)   \\
		 & \quad +b_0\l(\frac{1}{p}\l\|\nabla y_0^+\r\|_p^p+\frac{1}{ p}\l\|\nabla y_0^-\r\|_p^p+\frac{1}{q}\l\|\nabla y_0^+\r\|_{q,\mu}^q+\frac{1}{ q}\l\|\nabla y_0^-\r\|_{q,\mu}^q\r)^{\vartheta-1}  \\
		 & \quad \times \Bigg [ \left(\frac{1}{p\vartheta}-\frac{1}{q\vartheta}\right) \left(\l\|\nabla y_0^+\r\|_p^p+\l\|\nabla y_0^-\r\|_p^p \right)  \\
		 & \qquad \quad  +\left(\frac{1}{q\vartheta}-\frac{1}{q\vartheta}\right) \left(\l\|\nabla y_0^+\r\|_{q,\mu}^q+\l\|\nabla y_0^-\r\|_{q,\mu}^q \right)\Bigg]  \\
		 & \quad +\int_\Omega \left(\frac{1}{q\vartheta}f(x,y_0^+)y_0^+-F(x,y_0^+)\right)\,\mathrm{d} x   \\
		 & \quad
		+\int_\Omega \left(\frac{1}{q\vartheta}f(x,-y_0^-)(-y_0^-)-F(x,- y_0^-)\right)\,\mathrm{d} x   \\
		 & \leq \liminf_{n\to \infty} \left(\varphi\left(y_n^+-y_n^-\right)
		-\frac{1}{q\vartheta} \left\langle \varphi'\left(y_n^+-y_n^-\right),y_n^+-y_n^- \right\rangle \right)=m_0,
	\end{align*}
	where in the last line we first use the weak lower semicontinuity of $\|\cdot\|_p$ and $\|\cdot\|_{q,\mu}$ together with \eqref{prop_inf_achieved_2} and \eqref{prop_inf_achieved_3} and then rearrange the terms inside the limit. Therefore $\alpha_{y_0}=\beta_{y_0}=1$ and so the infimum $m_0$ is achieved by the function $y_0^+-y_0^-$.
\end{proof}

\begin{proposition}
	\label{proposition_sign-changing-solution}
	Let hypotheses \eqref{H1}--\eqref{H3} be satisfied and let $y_0\in\mathcal{M}$ be such that $\varphi(y_0)=m_0$. Then $y_0$ is a critical point of $\varphi$. In particular, $y_0$ is a least energy sign-changing solution of problem \eqref{problem}.
\end{proposition}

\begin{proof}
	We argue via contradiction and suppose that $\varphi'(y_0)\neq 0$. Then there exist  $\lambda, \delta_0 > 0$ such that
	\begin{align*}
		\|\varphi'(u)\|_{*} \geq \lambda, \quad
		\text{for all } u \in W^{1,\mathcal{H}}_0(\Omega) \text{ with } \|u-y_0\| < 3 \delta_0.
	\end{align*}
	Let $C$ be an embedding constant for $ W^{1, \mathcal{H} }_0 ( \Omega ) \to L^{p}(\Omega)$. Since $y_0^+ \neq 0 \neq y_0^-$, we have for any $v \in  W^{1, \mathcal{H} }_0 ( \Omega )$
	\begin{align*}
		\|y_0-v\|\geq C^{-1} \|y_0-v\|_p
		\geq
		\begin{cases}
			C^{-1} \|y_0^-\|_p, & \text{if } v^- = 0,  \\
			C^{-1} \|y_0^+\|_p, & \text{if } v^+ = 0.
		\end{cases}
	\end{align*}
	Now we take $\delta_1$ such that
	\begin{align*}
		\delta_1 \in \left(0,\min \left\{C^{-1} \|y_0^-\|_p,C^{-1} \|y_0^+\|_p\right\}\right).
	\end{align*}
	Therefore, for any $v \in  W^{1, \mathcal{H} }_0 ( \Omega )$ with $\|y_0-v\|< \delta_1$ it holds $v^+ \neq 0 \neq v^-$.

	Let $\delta = \min \{ \delta_0, \delta_1 / 2\}$. Since $(\alpha,\beta) \mapsto \alpha y_0^+ - \beta y_0^-$ is a continuous mapping from $[0,\infty)\times [0,\infty)$ into $ W^{1, \mathcal{H} }_0 ( \Omega )$, there exists $0 < \tau < 1$ such that for all $\alpha,\beta \geq 0$ with $\max \{|\alpha - 1|, |\beta - 1| \} < \tau$, we have
	\begin{align*}
		\left\|\alpha y_0^+ - \beta y_0^- - y_0 \right\| < \delta.
	\end{align*}
	Let $D = ( 1 - \tau, 1 + \tau) \times ( 1 - \tau, 1 + \tau)$. From Proposition \ref{proposition_nodal_unique-pair} we obtain for any $\alpha,\beta  \geq 0$ with $(\alpha,\beta) \neq  (1,1)$ that
	\begin{align}
		\label{Eq:ParametersEstimatedByInfimum}
		\varphi(\alpha y_0^+ - \beta y_0^-)
		<\varphi(y_0^+ - y_0^-)=\inf_{u \in \mathcal{M}} \varphi(u).
	\end{align}
	From this, we conclude that
	\begin{align*}
		\varrho = \max_{(\alpha,\beta) \in \partial D} \varphi(\alpha y_0^+ - \beta y_0^-)
		<\varphi(y_0^+ - y_0^-)=\inf_{u \in \mathcal{M}} \varphi(u).
	\end{align*}

	Now we are going to apply the quantitative deformation lemma given in Lemma \ref{Le:DeformationLemma} with
	\begin{align*}
		S = B (y_0, \delta), \quad c = \inf_{u \in \mathcal{M}} \varphi(u), \quad \varepsilon = \min \left\lbrace \frac{c - \varrho}{4}, \frac{\lambda \delta}{8} \right\rbrace , \quad \delta \text{ be as defined above.}
	\end{align*}
	Note that $S_{2 \delta} = B (y_0, 3 \delta)$. Then, by the choice of $\varepsilon$, the assumptions of Lemma \ref{Le:DeformationLemma} are satisfied. Therefore, a mapping $\eta$ with the properties stated in the lemma exists. Due to the choice of $\varepsilon$ we get
	\begin{align}
		\label{Eq:2epsAtBoundary}
		\varphi(\alpha y_0^+ - \beta y_0^-)
		\leq \varrho + c - c
		< c - \left( \frac{c - \varrho}{2} \right)
		\leq c - 2 \varepsilon
	\end{align}
	for all $(\alpha,\beta) \in \partial D$.

	Let us now define the mappings $H\colon [0,\infty)\times [0,\infty)\to  W^{1, \mathcal{H} }_0 ( \Omega )$, $\Pi\colon [0,\infty)\times [0,\infty)\to\R^2$ by
	\begin{align*}
		H(\alpha,\beta)    & =\eta\l(1,\alpha y_0^+-\beta y_0^-\r)   \\
		\Pi(\alpha,\beta ) & =\l(\l\langle \varphi'\l(H(\alpha,\beta)\r),H^+(\alpha,\beta)\r\rangle,\l\langle \varphi'\l(H(\alpha,\beta)\r),-H^-(\alpha,\beta)\r\rangle\r).
	\end{align*}
	It is clear that $H$ is continuous because of the continuity of $\eta$ and the differentiability of $\varphi$ implies that $\Pi$ is continuous. From Lemma \ref{Le:DeformationLemma} \textnormal{(i)} along with \eqref{Eq:2epsAtBoundary}, we have that $H(\alpha,\beta) = \alpha y_0^+ - \beta y_0^-$ for all $(\alpha,\beta) \in \partial D$ and
	\begin{align*}
		\Pi(\alpha,\beta) = \left( \langle \varphi'(\alpha y_0^+-\beta y_0^-) , \alpha y_0^+ \rangle , \langle \varphi'(\alpha y_0^+-\beta y_0^-) , - \beta y_0^- \rangle \; \right).
	\end{align*}
	Now, using the information on the derivatives from Proposition \ref{proposition_nodal_unique-pair}, we see that we have the componentwise inequalities
	\begin{align*}
		 & \Pi_1 (1 - \tau,t) > 0 > \Pi_1 (1 + \tau,t),  \\
		 & \Pi_2 (t, 1 - \tau) > 0 > \Pi_2 (t,1 + \tau) \quad
		\text{for all } t \in [1 - \tau, 1 + \tau],
	\end{align*}
	where $\Pi=(\Pi_1,\Pi_2)$.
	Then, by the Poincar\'{e}-Miranda existence theorem given in Theorem \ref{theorem-poincare-miranda} applied to $d(\alpha,\beta) = - \Pi(1 + \alpha, 1 + \beta)$, we find $(\alpha_0,\beta_0) \in D$ such that $\Pi(\alpha_0,\beta_0) = 0$ which can be equivalently written as
	\begin{align*}
		\langle \varphi'(H(\alpha_0 , \beta_0)) , H^+(\alpha_0 , \beta_0) \rangle = 0 = \langle \varphi'(H(\alpha_0 , \beta_0)) , -H^- (\alpha_0 , \beta_0) \rangle.
	\end{align*}
	Taking Lemma \ref{Le:DeformationLemma} \textnormal{(iv)} and the choice of $\tau$ into account, we obtain
	\begin{align*}
		\left\|H(\alpha_0,\beta_0)-y_0 \right\| \leq 2 \delta \leq \delta_1.
	\end{align*}
	However, by the choice of $\delta_1$, this leads to
	\begin{align*}
		H^+(\alpha_0 , \beta_0) \neq 0 \neq - H^-(\alpha_0 , \beta_0).
	\end{align*}
	Therefore, $H(\alpha_0,\beta_0) \in \mathcal{M}$. But, from Lemma \ref{Le:DeformationLemma} \textnormal{(ii)}, the choice of $\tau$ and \eqref{Eq:ParametersEstimatedByInfimum}, it follows that $\varphi(H(\alpha_0,\beta_0) ) \leq c - \varepsilon$, a contradiction. Thus, $y_0$ is a critical point of $\varphi$ and so a least energy sign-changing solution to our problem \eqref{problem}, see Proposition \ref{proposition_infimum_achieved}.
\end{proof}

The proof of Theorem \ref{theorem_main_result} follows now from Propositions \ref{proposition_constant_sign_solutions} and \ref{proposition_sign-changing-solution}.
\section*{Acknowledgment}

The first author was funded by the Deutsche Forschungsgemeinschaft (DFG, German Research Foundation) under Germany's Excellence Strategy - The Berlin Mathematics Research Center MATH+ and the Berlin Mathematical School (BMS) (EXC-2046/1, project ID: 390685689). The second author thanks the University of Technology Berlin for the kind hospitality during a research stay in September 2023.

\end{document}